\documentclass[11pt,a4paper]{amsart}
\usepackage{amsmath}
\usepackage{amssymb,xspace}
\usepackage{amstext}
\theoremstyle{plain}
\usepackage{amsbsy,amssymb,amsfonts,latexsym}
\usepackage[utf8]{inputenc}
\usepackage{xcolor}

\usepackage{enumerate}
\usepackage{graphics}

\date{\today}
\title[Convergence and divergence of wavelet series]{Convergence and divergence of wavelet series: multifractal aspects}
\author{Fr\'ed\'eric Bayart}
\thanks{The author was partially supported by the grant ANR-17-CE40-0021 of the French National Research Agency ANR (project Front)}
\address{Université Clermont Auvergne, CNRS, LMBP, F-63000 Clermont–Ferrand, France.}
\email{Frederic.Bayart@uca.fr}



\newcommand{\veps}{\varepsilon}

\def\RR{\mathbb R}

\def\NN{\mathbb N}
\def\ZZ{\mathbb Z}

\def\card{\textrm{card}}
\def\dboxsup{\overline{\dim_{\rm{B}}}}
\def\dimh{\dim_{\mathcal H}}
\def\dimp{\dim_{\mathcal P}}
\def\besov{B_p^{s,q}(\RR^d)}
\def\sobolev{W^{p,s}(\RR^d)}

\newtheorem{theorem}{Theorem}[section]

\newtheorem{lemma}[theorem]{Lemma}

\newtheorem{proposition}[theorem]{Proposition}

{\theoremstyle{definition}}

{\theoremstyle{definition}}

{\theoremstyle{definition}}

{\theoremstyle{definition}}

{\theoremstyle{definition}\newtheorem{remark}[theorem]{Remark}}

\newtheorem{question}[theorem]{Question}

\newtheorem*{AUBRY}{Theorem (Aubry)}
\newtheorem*{BH}{Theorem (Bayart-Heurteaux)}
\newtheorem*{EJ}{Theorem (Esser-Jaffard)}

\begin{document}

\begin{abstract}
We study the convergence and divergence of the wavelet expansion of a function in a Sobolev or a Besov space from a multifractal point of view. 
In particular, we give an upper bound for the Hausdorff and for the packing dimension of the set of points where the expansion converges (or diverges) 
at a given speed, and we show that, generically, these bounds are optimal.
\end{abstract}

\maketitle

\section{Introduction}
\subsection{Wavelet expansion}
This paper deals with the local behaviour of the wavelet expansion of a given function. Recall that an orthogonal multiresolution analysis (MRA) with scaling function $\varphi$ is a collection of subspaces $(V_j)_{j\in\mathbb Z}$ of $L^2(\RR^d)$
such that
\begin{enumerate}
 \item $V_j\subset V_{j+1}$ for all $j\in\mathbb Z$;
 \item $\bigcap_{j\in\ZZ}V_j=\{0\}$;
 \item $\bigcup_{j\in\ZZ}V_j$ is dense in $L^2(\RR^d)$;
 \item $f(x)\in V_j\iff f(2x)\in V_{j+1}$;
 \item $\varphi\in V_0$ and its integer translates $(\varphi(x-k))_{k\in\ZZ^d}$ form an orthonormal basis for $V_0$.
\end{enumerate}
The orthogonal projection $P_j$ on $V_j$  is called the partial reconstruction operator of order $j$. We can 
associate to the MRA a wavelet basis, namely a collection $\psi^{(i)}$, $i=1,\dots,2^d-1$, of functions in $L^2(\RR^d)$ such that the functions 
$$2^{dj/2}\psi^{(i)}(2^j \cdot -k)\textrm{ for }i\in\{1,\dots,2^d-1\},\ j\in\ZZ,\ k\in\ZZ^d$$
form an orthonormal basis of $L^2(\RR^d)$. The reconstruction operator $P_j$, for $j\geq 0$, can also be expressed by 
$$P_jf(x)=\sum_{k\in\ZZ^d}\langle f,\varphi(\cdot-k)\rangle\varphi(x-k)+\sum_{l<j}\sum_{i=1}^{2^d-1}\sum_{k\in\ZZ^d}2^{dl}\langle f,\psi_{l,k}^{(i)}\rangle\psi_{l,k}^{(i)}(x)$$
where $\psi_{l,k}^{(i)}=\psi^{(i)}(2^l\cdot-k)$.

Wavelet expansions have many remarkable properties. They provide unconditional basis of many function spaces, like $L^p$-spaces ($1<p<+\infty$), Sobolev spaces and Besov spaces. In particular, $(P_j f)$ converges
to $f$ with respect to the corresponding norm. In this paper, we are concerned with the pointwise convergence or divergence of $(P_jf(x))$.

\subsection{Known results}\ 

\noindent $\blacktriangleright${\bf Convergence.} This question was already investigated in many papers. In \cite{KKR94}, the authors show that $(P_jf(x))$ converges almost everywhere for all $f\in L^p$ 
($p\geq 1$): the convergence holds at all Lebesgue points of $f$. When $f$ is continuous, the convergence is locally uniform
(see \cite{Wal95}) and in smooth Sobolev spaces, one can even control $\|f-P_j f\|_\infty$ (see \cite{KR01}).

\smallskip

\noindent $\blacktriangleright${\bf Aubry results.} In \cite{Aubry06}, Aubry is the first to study the set of points 
where $(P_j f(x))$ diverges. In his paper, he answers several natural questions: can we say something on the speed of divergence of $(P_jf(x))$? Can we say something on the size of the sets of $x\in\RR^d$ such that $(P_jf(x))$ diverges at a given speed? To state Aubry's 
result, it is convenient to introduce the following sets, for $\beta>0$ and $f\in L^p(\RR^d)$:
\begin{eqnarray*}
\mathcal E^-(\beta,f)&=&\left\{x\in\RR^d;\ \limsup_j \frac{\log |P_jf(x)|}{j\log 2}\geq \beta\right\}\\
E^-(\beta,f)&=&\left\{x\in\RR^d;\ \limsup_j \frac{\log |P_jf(x)|}{j\log 2}= \beta\right\}\\
\mathcal E^+(\beta,f)&=&\left\{x\in\RR^d;\ \liminf_j \frac{\log |P_jf(x)|}{j\log 2}\geq \beta\right\}\\
E^+(\beta,f)&=&\left\{x\in\RR^d;\ \liminf_j \frac{\log |P_jf(x)|}{j\log 2}= \beta\right\}\\
E(\beta,f)&=&\left\{x\in\RR^d;\ \lim_j \frac{\log |P_jf(x)|}{j\log 2}= \beta\right\}.
\end{eqnarray*}
In what follows, $\dimh(E)$ will denote the Hausdorff dimension of $E$ and $\dimp(E)$ its packing dimension.
With this terminology, Aubry's theorem reads:
\begin{AUBRY}
Let $f\in L^p(\RR^d)$, $1<p<+\infty$ and $\beta> 0$. Then $\dimh\big(\mathcal E^-(\beta,f)\big)\leq d-\beta p$. 
Conversely, if we are working with the Haar wavelet, given a set $E\subset\RR$ such that $\dimh(E)<1-\beta p$, there exists $f\in L^p(\RR)$ such that $E\subset\mathcal E^{-}(\beta,f)$. 
\end{AUBRY}
Strictly speaking, Aubry's result was formulated for periodized wavelets, but his proof carries on to our context.

\smallskip

\noindent $\blacktriangleright${\bf Bayart-Heurteaux results.} In \cite{BH17}, 
as an application of the general framework developed there, the authors improve the results of Aubry in two directions. First, they provide a bound for the dimension of
$\mathcal E^+(\beta,f)$
involving the packing dimension. Second, in the spirit of \cite{Jaf00} for the study of the local H\"older exponent and of \cite{BAYHEUR1} for the divergence of Fourier series,
they show that we can construct functions whose behaviour is multifractal with respect to the divergence of their wavelet expansion.

\begin{BH}
 Assume that we are working with the Haar wavelet. 
\begin{itemize}
\item[(i)] For all $\beta\in(0,1/2]$ and all $f\in L^2(\mathbb R)$, 
$$\dimp\big(\mathcal E^+(\beta,f)\big)\leq 1-2\beta;$$
\item[(ii)] For all functions $f$ in a residual subset of $L^2(\mathbb R)$, for all $\beta\in (0,1/2]$, 
$$\dimh\big(E^-(\beta,f)\big)=1-2\beta;$$
\item[(iii)] There exists a function $f\in L^2(\mathbb R)$ such that, for all $\beta\in(0,1/2]$, 
$$\dimh\big(E(\beta,f)\big)=\dimp\big(E(\beta,f)\big)=1-2\beta.$$
\end{itemize}
\end{BH}
It should be pointed out that, to deduce this result from the general results proved in \cite{BH17}, specific properties of the Haar basis were needed, in
particular the positivity of the projections $P_j$. These specific properties were also important for the proof of the second half of Aubry's theorem.

\smallskip

\noindent $\blacktriangleright${\bf Esser-Jaffard results.} Very recently, Esser and Jaffard undertake in \cite{EJ17} a multifractal analysis of the divergence
of general wavelet series belonging to Besov spaces $B_{p}^{s,q}(\RR^d)$. From now on, wavelets are assumed to be smooth enough, say, with at least derivatives up to order $\lfloor s\rfloor+1$
having fast decay. To overcome the difficulty of working in a general context, Esser and Jaffard do not study the behaviour of $|P_jf(x)|$, but that of the coefficients
$2^{dj}\langle f,\psi_{j,k}^{(i)}\rangle\psi_{j,k}^{(i)}(x)$. More precisely, let us define, for $\beta\in\RR$, 
\begin{eqnarray*}
 \mathcal F^{-}(\beta,f)&=&\left\{x\in\RR^d;\ \limsup_j \frac{\log \sup_{i,k}|2^{dj}\langle f,\psi_{j,k}^{(i)}\rangle\psi_{j,k}^{(i)}(x)|}{j\log 2}\geq\beta\right\}\\
 F^{-}(\beta,f)&=&\left\{x\in\RR^d;\ \limsup_j \frac{\log \sup_{i,k}|2^{dj}\langle f,\psi_{j,k}^{(i)}\rangle\psi_{j,k}^{(i)}(x)|}{j\log 2}=\beta\right\}.
\end{eqnarray*}
It can be easily observed (see \cite[Proposition 2.1]{EJ17}) that, for all $0<\gamma<\beta$, $\mathcal E^-(\beta,f)\subset \mathcal F^-(\gamma,f)$ (heuristically speaking,
if the sum is large, at least one of the coefficients
should be large). 

With this terminology, we can state their main theorem as follows.
\begin{EJ}
 Let $s\geq 0$, $p,q\in (0,+\infty)$. 
 \begin{itemize}
  \item[(i)] For all $f\in \besov$, for all $\beta\in \left[-s,\frac dp-s\right]$, $\dimh\big(\mathcal F^-(\beta,f)\big)\leq d-sp-\beta p$.
  \item[(ii)] For all $f$ in a residual and prevalent subset of $\besov$, for all $\beta\in \left[-s,\frac dp-s\right]$, $\dimh\big(F^-(\beta,f)\big)=d-sp-\beta p$.
 \end{itemize}
\end{EJ}
Prevalence is an extension of the notion of almost everywhere in infinite-dimensional vector spaces. We shall use only the following properties (which appear e.g. in \cite{Chr72} or in \cite{HSY})
where $X$ is any Banach space:
\begin{itemize}
 \item the countable intersection of prevalent subsets of $X$ remains a prevalent subset of $X$;
 \item if $A\subset B\subset X$ and $A$ is a prevalent subset of $X$, then $B$ is a prevalent subset of $X$;
 \item in order to prove that $Y\subset X$ is prevalent, it is enough to find a finite-dimensional subspace $V$ of $X$ such that, 
 for all $f\in X$, for almost all $v\in V$ (with respect to the Lebesgue measure on $V$), $f+v\in Y$. In that case, we say that $Y$ is $\dim V$-prevalent.
\end{itemize}

\smallskip

\noindent $\blacktriangleright${\bf Our results.} In the present paper, we come back to the study of the divergence of $(P_jf(x))$, which seems more delicate since compensations can come into play. We also investigate the $\liminf$ and $\lim$ sets, namely $\mathcal E^+(\beta,f)$, $E^+(\beta,f)$ and $E(\beta,f)$, which need very careful constructions since we want to control $(P_j)$ for all $j$ and not only for some $j$. Our first result is a full generalization of the results of Aubry and Bayart/Heurteaux to all wavelet basis with compact support and to Besov and Sobolev spaces
admitting functions whose wavelet expansion diverges at some point (namely when $d-sp>0$). 
\begin{theorem}\label{thm:main1}
 Let $s\geq 0$, $p,q\in [1,+\infty)$ and $X=\besov$ or $X=\sobolev$. Assume that the wavelets have compact support.
 \begin{itemize}
  \item[(i)] For all $f\in X$, for all $\beta\in \left(0,\frac dp-s\right]$,
  \begin{eqnarray*}
   \dimh\big(\mathcal E^-(\beta,f)\big)&\leq&d-sp-\beta p\\
      \dimp\big( E(\beta,f)\big)&\leq&d-sp-\beta p.
  \end{eqnarray*}
\item[(ii)] For all  $f$ in a residual and prevalent subset of $X$, for all $\beta\in\left(0,\frac dp-s\right]$, 
$$\dimh\big(E^{-}(\beta,f)\big)=d-sp-\beta p.$$ 
\item[(iii)] There exists $f\in X$ such that for all $\beta\in\left(0,\frac dp-s\right]$,
$$\dimh\big(E(\beta,f)\big)=\dimp\big(E(\beta,f)\big)=d-sp-\beta p.$$
 \end{itemize}
\end{theorem}

Theorem \ref{thm:main1} does not cover all natural cases. Indeed, in Besov spaces, wavelet series are convergent at many points (and 
even at all points if $d-sp<0$). For such a point, one is interested in the speed of decay to zero of the remainder
$$R_j f(x)=\sum_{l\geq j}\sum_{i=1}^{2^d-1}\sum_{k\in\ZZ^d}2^{dl}\langle f,\psi_{l,k}^{(i)}\rangle\psi_{l,k}^{(i)}(x).$$
This motivates us to introduce, for $\beta<0$, the following sets:
\begin{eqnarray*}
\mathcal E^-(\beta,f)&=&\left\{x\in\RR^d;\ P_jf(x)\textrm{ converges and }\limsup_j \frac{\log |R_jf(x)|}{j\log 2}\geq \beta\right\}\\
E^-(\beta,f)&=&\left\{x\in\RR^d;\  P_jf(x)\textrm{ converges and }\limsup_j \frac{\log |R_jf(x)|}{j\log 2}= \beta\right\}\\
\mathcal E^+(\beta,f)&=&\left\{x\in\RR^d;\  P_jf(x)\textrm{ converges and }\liminf_j \frac{\log |R_jf(x)|}{j\log 2}\geq \beta\right\}\\
E^+(\beta,f)&=&\left\{x\in\RR^d;\  P_jf(x)\textrm{ converges and }\liminf_j \frac{\log |R_jf(x)|}{j\log 2}= \beta\right\}\\
E(\beta,f)&=&\left\{x\in\RR^d;\  P_jf(x)\textrm{ converges and }\lim_j \frac{\log |R_jf(x)|}{j\log 2}= \beta\right\}.
\end{eqnarray*}
We get the following version of Theorem \ref{thm:main1} for these convergence sets.

\begin{theorem}\label{thm:main2}
 Let $s\geq 0$, $p,q\in [1,+\infty)$ and $X=\besov$ or $X=\sobolev$. Assume that the wavelets have compact support.
 \begin{itemize}
  \item[(i)] For all $f\in X$, for all $\beta\in \left[-s,\min\left(0,\frac dp-s\right)\right]\backslash\{0\}$,
  \begin{eqnarray*}
   \dimh\big(\mathcal E^-(\beta,f)\big)&\leq&d-sp-\beta p\\
      \dimp\big( E(\beta,f)\big)&\leq&d-sp-\beta p.
  \end{eqnarray*}
\item[(ii)] For all  $f$ in a residual and prevalent subset of $X$, for all $\beta\in\left[-s,\min\left(0,\frac dp-s\right)\right]\backslash\{0\}$, 
$$\dimh\big(E^{-}(\beta,f)\big)=d-sp-\beta p.$$ 
\item[(iii)] There exists $f\in X$ such that for all $\beta\in\left[-s,\min\left(0,\frac dp-s\right)\right]\backslash\{0\}$,
$$\dimh\big(E(\beta,f)\big)=\dimp\big(E(\beta,f)\big)=d-sp-\beta p.$$
 \end{itemize}
\end{theorem}

If we look carefully at Part (i) of Theorems \ref{thm:main1} and \ref{thm:main2} and if we compare it with Part (i) of Bayart/Heurteaux theorem, or with the standard 
inequality on the local dimension of measures, we observe that
we only get an estimation of the packing dimension of $E(\beta,f)$ whereas it would be natural to expect the stronger inequality $\dimp\big(\mathcal E^+(\beta,f)\big)\leq d-sp-\beta p$.
Surprizingly, when $s>0$, this inequality is not satisfied by all functions when $d-sp>0$ whereas it is satisfied by all functions if $d-sp<0$. 
\begin{theorem}\label{thm:badpacking}
 Assume that the wavelets have compact support. 
 \begin{itemize}
  \item[(i)]If $s>0$, $d=1$ and $1-sp>0$, for all $\beta\in \left(-s,\frac 1p- s\right)\backslash\{0\}$, there exists $f\in B_p^{s,1}(\mathbb R)$ such that 
 \begin{equation*}
  \dimp\left(\mathcal E^+(\beta,f)\right)>1-sp-\beta p.
 \end{equation*}
  \item[(ii)]If $d-sp<0$, for all $\beta\in \left(-s,\frac dp- s\right)\backslash\{0\}$, for all $f\in B_p^{s,\infty}(\RR^d)$,
  \[\dimp\big(\mathcal E^+(\beta,f)\big)\leq d-sp-\beta p.\]
 \end{itemize}
\end{theorem}

The paper is organized as follows. In Section \ref{sec:preliminaries}, we introduce definitions and notations used throughout the paper.
Section \ref{sec:dimension} contains the proof of part (i) of Theorems \ref{thm:main1} and \ref{thm:main2} and even more: we do not need
the assumption that the wavelets have compact support here. 
Section \ref{sec:existence} is devoted to the proof of the remaining parts of these theorems. The main difficulty that we have to overcome is the nonpositivity of the projections $P_j$. We tackle it by the construction of a Cantor set
where we control the behaviour of the wavelets. In Section \ref{sec:packing}, we turn to a detailed study of the packing dimension of the
sets $\mathcal E^+(\beta,f)$. Here too, we need to construct a Cantor set with special properties to be able to define a function $f\in B_p^{s,1}(\RR)$
such that $\dimp \big(\mathcal E^+(\beta,f)\big)>1-sp-\beta p$. The last section contains additional remarks.

\subsection{Notations}
We shall use the following notations. For $p\in [1,+\infty]$, $p^*$ denotes its conjugate exponent, $1/p+1/p^*=1$.
The letter $C$ will denote
a constant (which usually depends on the parameters $p,\ q,\ s,\ d$ and on the wavelets $\psi^{(i)}$, but does not depend  on the level $j$ of the projection),
whose value may change from line to line. To emphasize that $C$ depends on $A$, 
we occasionaly write $C_A$. 

\section{Preliminaries}\label{sec:preliminaries}
\subsection{Dyadic cubes}
We shall index wavelets using dyadic cubes. For $k=(k_1,\dots,k_d)\in\ZZ^d$ and $j\geq 0$, $\lambda=(j,k)$ will denote the dyadic cube of the $j$-th generation
$$\lambda=(j,k):=\left[\frac{k_1}{2^j},\frac{k_1+1}{2^j}\right)\times\cdots\times \left[\frac{k_d}{2^j},\frac{k_d+1}{2^j}\right).$$
We will index wavelets and wavelet coefficients by $(i,j,k)$ or by $(i,\lambda)$, writing indifferently $\psi_\lambda^{(i)}$ or $\psi_{j,k}^{(i)}$. Furthermore, $\Lambda_j$ will denote the set of dyadic cubes of the $j$-th generation.
Any element $x\in\RR^d$ belongs to a unique $\lambda\in\Lambda_j$ which we will denote by $\lambda_j(x)$.
We take for norm on $\RR^d$ the supremum norm, so that the diameter of a dyadic cube of $\Lambda_j$ is exactly $2^{-j}$.

\subsection{Besov and Sobolev spaces}

We shall use the following definition for Besov spaces. We start with a MRA with scaling function $\varphi$ and wavelet basis $(\psi_\lambda^{(i)})$. Let $f\in L^p(\RR^d)$ and define, for $k\in\ZZ^d$, $i\in\{1,\dots,2^d-1\}$ and $\lambda$ a dyadic cube,
$$C_k=\int_{\RR^d}\overline{\varphi(x-k)}f(x)dx,$$
$$c_\lambda^{(i)}=\int_{\RR^d}2^{dj}\overline{\psi_\lambda^{(i)}(x)}f(x)dx.$$
Then we say that $f$ belongs to the Besov space $B_p^{s,q}(\RR^d)$ ($s\geq 0$, $p\in(0,+\infty]$, $q\in (0,+\infty]$) if $(C_k)$ belongs to $\ell^p$ and if, setting for all $j\geq 1$
$$\veps_j=2^{\left(s-\frac dp\right)j}\left(\sum_i \sum_{\lambda\in\Lambda_j}|c_\lambda^{(i)}|^p\right)^{1/p}$$
then the sequence $(\veps_j)$ belongs to $\ell^q$ (we shall use the $L^\infty$ normalization for wavelets).
The norm of $f$ in $B_p^{s,q}(\RR^d)$ is then defined as the sum of the $\ell^p$-norm of $(c_k)$ and the $\ell^q$-norm of $(\veps_j)$. 
 When the wavelets are smooth enough, 
an assumption that we make throughout the paper, this definition matches the classical definition of Besov spaces (see \cite{Mey90}). We also observe that we immediately get that, for all $f\in B_p^{s,\infty}(\RR^d)$ and all $\lambda\in\Lambda_j$, $|c_\lambda^{(i)}|\leq C2^{\left(\frac dp-s\right)j}$. 

Besov and Sobolev spaces are very close. It is well known (see for instance \cite{BerLof}) that 
$$B_p^{s,1}(\RR^d)\subset W^{p,s}(\RR^d)\subset B_p^{s,\infty}(\RR^d),$$
where $W^{p,s}(\RR^d)$ stands for the usual Sobolev space. We shall use these inclusions by producing saturating functions in $B_p^{s,1}(\RR^d)$ and by estimating the dimension of the level sets 
for functions in $B_p^{s,\infty}(\RR^d)$.

%

\subsection{Wavelets}
Throughout this work, we shall assume that the wavelets have fast decay, namely that, for all $N\geq 0$, there exists
a constant $C_N>0$ such that, for all $i=1,\dots,2^d-1$ and all $x\in\mathbb R^d$, 
\begin{equation}\label{eq:fastdecaywavelets}
\left|\psi^{(i)}(x)\right|\leq \frac{C_N}{(1+\|x\|)^N}.
\end{equation}

We shall use several times the following lemmas, which are easy consequences of \eqref{eq:fastdecaywavelets}.

\begin{lemma}\label{lem:wavelet1}
 There exists $C>0$ such that, for all $j\in \mathbb N$ and all $x\in\mathbb R^d$, 
 $$\sum_i \sum_{\lambda\in\Lambda_j} |\psi_\lambda^{(i)}(x)|\leq C.$$
\end{lemma}

\begin{lemma}\label{lem:wavelet2}
 Let $\veps>0$ and $\kappa>0$. There exists $C_{\veps,\kappa}$ such that, for all $x\in\RR^d$, for all $j\in\NN$, 
 $$\sum_i \sum_{\substack{\lambda=(j,k); \\ \|2^j x-k\|\geq 2^{\veps j}}} |\psi_{\lambda}^{(i)}(x)|\leq C_{\veps,\kappa} 2^{-\kappa j}.$$
\end{lemma}
\begin{proof}
Lemma \ref{lem:wavelet1} follows immediately from \eqref{eq:fastdecaywavelets} with $N\geq d+1$ and standard calculus. To prove Lemma \ref{lem:wavelet2}, we write for $\|2^j x-k\|\geq 2^{\veps j}$, 
$$|\psi_\lambda^{(i)}(x)|\leq \frac{C_N}{\left(1+\|2^j x-k\|\right)^{N/2}}2^{-\veps N/2}$$
and we choose $N\geq \max(2d+2,2\kappa/\veps)$.
\end{proof}

Lemma \ref{lem:wavelet1} in turn easily implies that, for all $f\in B_p^{s,\infty}(\RR^d)$, for all $x\in\RR^d$, for all $j\in\NN$, $|P_j f(x)|\leq C2^{\left(\frac dp-s\right)j}$ if $d-sp>0$ and $|R_jf(x)|\leq 2^{\left(\frac dp-s\right)j}$
if $d-sp<0$, which justifies the restriction $\beta\leq \frac dp-s$ in Theorems \ref{thm:main1} and \ref{thm:main2}. 
They are also useful to prove the following result vthat quantifies how close are $P_j f(x)$ and $P_j f(y)$ if $x$ and $y$ are close.
We define, for $l\geq 1$,
\[Q_l f(x)=\sum_{i=1}^{2^d-1}\sum_{k\in\ZZ^d}2^{dl}\langle f,\psi_{l,k}^{(i)}\rangle\psi_{l,k}^{(i)}(x)=\sum_{i=1}^{2^d-1}\sum_{\lambda\in\Lambda_l}c_\lambda^{(i)}\psi_\lambda^{(i)}(x).\]

\begin{lemma}\label{lem:wavelet3}
 Let $s\geq 0$, $p,q\in[1,+\infty]$, $\beta\in\RR$. There exist $C_\beta>0$ and $\theta>0$ such that, for all $f\in B_p^{s,q}(\RR^d)$, for all $j\in\NN$, 
 for all $x,y\in\RR^d$ with $\|x-y\|<2^{-\theta j}$, $|Q_j f(x)-Q_j f(y)|\leq C_\beta 2^{\beta j}\|f\|_{B_{p}^{s,q}}$.
\end{lemma}
\begin{proof}
Let $f\in B_p^{s,q}(\RR^d)$ and let $x,y\in\RR^d$.
 By H\"older's inequality,
 \begin{align*}
  |Q_j f(x)-Q_jf(y)|&\leq \|f\|2^{-\left(s-\frac dp\right)j}\left(\sum_i\sum_{\lambda\in \Lambda_j}|\psi_\lambda^{(i)}(x)-\psi_\lambda^{(i)}(y)|^{p^*}\right)^{1/p^*}\\
  &\leq \|f\|2^{-\left(s-\frac dp\right)j}\sum_i\sum_{\lambda\in \Lambda_j} |\psi_\lambda^{(i)}(x)-\psi_\lambda^{(i)}(y)|\\
  &\leq \|f\|2^{-\left(s-\frac dp\right)j}\sum_i\left(\sum_{\lambda\in\Gamma_1}|\psi_\lambda^{(i)}(x)-\psi_\lambda^{(i)}(y)|+\sum_{\lambda\in\Gamma_2}|\psi_\lambda^{(i)}(x)-\psi_\lambda^{(i)}(y)|\right)
 \end{align*}
where $\Gamma_1=\{\lambda=(j,k)\in\Lambda_j;\ \|2^j x-k\|\geq 2^j\textrm{ and }\|2^j y-k\|\geq 2^j\}$ and $\Gamma_2=\Lambda_j\backslash \Gamma_1$.
By Lemma \ref{lem:wavelet2}, there exists some $C_\beta>0$ such that, for all $x,y\in\RR^d$, 
$$2^{-\left(s-\frac dp\right)j}\sum_i\sum_{\lambda\in\Gamma_1}|\psi_\lambda^{(i)}(x)-\psi_\lambda^{(i)}(y)|\leq C_\beta 2^{\beta j}.$$
On the other hand, since $\textrm{card}(\Gamma_2)\leq C 2^{dj}$,
\begin{align*}
 2^{-\left(s-\frac dp\right)j}\sum_i\sum_{\lambda\in\Gamma_2}|\psi_\lambda^{(i)}(x)-\psi_\lambda^{(i)}(y)|&\leq  C 2^{-\left(s-\frac dp\right)j}2^{dj}2^j\|x-y\|\\
 &\leq C2^{-\left(s-\frac dp-d-1+\theta\right)j}
\end{align*}
provided $\|x-y\|\leq 2^{-\theta j}$. A choice of $\theta>0$ large enough allows us to conclude.

\end{proof}

%
%


\section{Upper bounds for the dimension}\label{sec:dimension}
\subsection{Hausdorff dimension}
Our aim in this subsection is to prove the following proposition, which does not require that the wavelets have compact support.

\begin{proposition}\label{prop:dimh}
 Let $\beta\in \left[-s,\frac dp-s\right]\backslash\{0\}$ and $f\in B_{p}^{s,\infty}(\RR^d)$. Then $\dimh\big(\mathcal E^{-}(\beta,f)\big)\leq d-sp-\beta p$.
\end{proposition}
\begin{proof}
 We first observe that in the case $\beta> 0$, this is already known. This follows indeed from the inclusion $\mathcal E^{-}(\beta,f)\subset \mathcal F^{-}(\gamma,f)$ for
 all $\gamma<\beta$ and from the corresponding result of Esser and Jaffard. For $\beta\in\left(-s,\min\left(0,\frac dp-s\right)\right)$ (the result is trivial if $\beta=-s$),
 the inclusion is reversed and we need to provide a proof (inspired by that of \cite{EJ17}).
 Let $\gamma\in(-s,\beta)$ and $\veps>0$. For $j\in\mathbb N$, we define
 \begin{eqnarray*}
  \Gamma_{j,\gamma}&=&\left\{\lambda\in\Lambda_j;\ \exists i,\ |c_\lambda^{(i)}|\geq 2^{\gamma j}\right\}\\
  E_{j,\gamma,\veps}&=&\bigcup_{\lambda\in\Gamma_{j,\gamma}} \lambda+B\big(0,2^{-(1-\veps)j}\big)\\
  E_{\gamma,\veps}&=&\limsup_{j\to+\infty}E_{j,\gamma,\veps}.
 \end{eqnarray*}
 Since $f$ belongs to $B_p^{s,\infty}(\RR^d)$, the cardinal number of $\Gamma_{j,\gamma}$ is less than $C2^{(d-sp-\gamma p)j}$. Thus, $E_{j,\gamma,\veps}$ 
 is composed of at most $C 2^{(d-sp-\gamma p)j}$ cubes of width $C 2^{-(1-\veps)j}$. Using these cubes for $j$ large as a covering of $E_{\gamma,\veps}$ yields
 $$\dimh\big(E_{\gamma,\veps}\big)\leq\frac{d-sp-\gamma p}{1-\veps}.$$
 Letting $\gamma$ to $\beta$ and $\veps$ to 0, we get the conclusion if we prove that $\mathcal E^{-}(\beta,f)\subset E_{\gamma,\veps}$.
 Therefore, assume that $x\notin E_{\gamma,\veps}$. Let $J\in\NN$ be such that,
 for all $j\geq J$, $x\notin E_{j,\gamma,\veps}$. For $j\geq J$ one may write
 $$|R_j f(x)|\leq\sum_{l\geq j}\left(\sum_{\lambda\in \Lambda_l\backslash \Gamma_{l,\gamma}}\sum_i|c_\lambda^{(i)}|\cdot|\psi_\lambda^{(i)}(x)|+
 \sum_{\lambda\in \Gamma_{l,\gamma}}\sum_i|c_\lambda^{(i)}|\cdot|\psi_\lambda^{(i)}(x)|\right).$$
 Now, let $\lambda=(l,k)\in \Gamma_{l,\gamma}$. Since $x\notin E_{l,\gamma,\veps}$, $\|2^l x-k\|\geq 2^{\veps l}$. Moreover, $|c_\lambda^{(i)}|\leq C 2^{\left(\frac dp-s\right)l}$.
 Using Lemma \ref{lem:wavelet2}
 with a sufficiently large $\kappa$, we get
 $$\sum_{l\geq j}\sum_{\lambda\in \Gamma_{l,\gamma}}\sum_i |c_\lambda^{(i)}|\cdot |\psi_\lambda^{(i)}(x)|\leq C 2^{\gamma j}.$$
Furthermore, 
\begin{eqnarray*}
 \sum_{l\geq j}\sum_{\lambda\in\Lambda_l\backslash \Gamma_{l,\gamma}}\sum_i |c_\lambda^{(i)}|\cdot|\psi_\lambda^{(i)}(x)|&\leq&\sum_{l\geq j}\sum_{\lambda\in\Lambda_l}\sum_i2^{\gamma l}|\psi_\lambda^{(i)}(x)| \\
 &\leq&C2^{\gamma j}
\end{eqnarray*}
by Lemma \ref{lem:wavelet1}. Hence $x\notin\mathcal E^{-}(\beta,f)$.
\end{proof}

A small variant of the above proof implies the following result, which will be needed later.
\begin{proposition}\label{prop:dimhdv}
 Let $f\in B_p^{s,\infty}(\RR^d)$. Then 
 $$\dimh\big(\big\{x\in\RR^d;\ (P_j f(x))\textrm{ diverges}\big\}\big)\leq d-sp.$$
\end{proposition}
\begin{proof}
 Keeping the same notation, it suffices to observe that, for any $\gamma<0$ and any $\veps>0$, $(P_jf(x))$ converges provided $x\notin E_{\gamma,\veps}$.
\end{proof}

\subsection{Packing dimension}\ 
 We now prove the statement about the packing dimension (again, our proof does not require that the wavelets are compactly supported).
\begin{proposition}\label{prop:packing1}
 Let $\beta\in \left[-s,\frac dp-s\right]\backslash\{0\}$ and $f\in B_p^{s,\infty}(\RR^d)$. Then $\dimp\big( E(\beta,f)\big)\leq d-sp-\beta p$.
\end{proposition}
We need to introduce some notations. For $\lambda_0\in\Lambda$, $\veps>0$ and $l\in\mathbb N$, we denote
$$\Lambda_{l,\lambda_0,\veps}=\left\{\lambda\in\Lambda_l;\ \big(\lambda+B(0,2^{-(1-\veps)l})\big)\cap\lambda_0\neq\varnothing\right\}.$$
It is not difficult to observe that, $j,l,\veps$ being kept fixed, any $\lambda\in\Lambda_l$ belongs to at most $C_d2^{d(j-l)+\veps dl}$ different sets $\Lambda_{l,\lambda_0,\veps}$
for $\lambda_0$ describing $\Lambda_j$ and that, for a fixed $\lambda_0\in\Lambda_j$, 
$$\textrm{card}(\Lambda_{l,\lambda_0,\veps})\leq C_d\left(2^{\veps dl}+2^{d(l-j)}\right).$$
The cubes which are not in $\Lambda_{l,\lambda_0,\veps}$ are cubes with few interaction with $\lambda_0$. In particular, if $x\in\lambda_0$ and $\lambda=(l,k)\notin\Lambda_{l,\lambda_0,\veps}$, then
\begin{equation}
 \|2^l x-k\|\geq 2^{\veps l}. \label{eq:packing1}
\end{equation}
Let us also set
\begin{eqnarray*}
 \|f\|_{l,\lambda_0,\veps}&=&\left(\sum_i\sum_{\lambda\in\Lambda_{l,\lambda_0,\veps}} |c_\lambda^{(i)} 2^{\left(s-\frac dp\right)l}|^p\right)^{1/p}\\
 \|f\|_{l}&=&\left(\sum_i\sum_{\lambda\in\Lambda_{l}} |c_\lambda^{(i)} 2^{\left(s-\frac dp\right)l}|^p\right)^{1/p}\\
\end{eqnarray*}
It follows from the above discussion that, for all $\lambda_0\in\Lambda_j$ and for all $l\in\mathbb N$, 
$$\left(\sum_{\lambda_0\in \Lambda_j}\|f\|_{l,\lambda_0,\veps}^p\right)^{1/p}\leq  C_d 2^{d(j-l)+\veps dl} \|f\|_l.$$

Our starting point is to say that if we control the behaviour of all $P_jf(x)$ or all $R_j f(x)$,
then we control the behaviour of at least one $Q_l f(x)$ for $l$ close to $j$.

\begin{lemma}\label{lem:packing0}
 Let $\beta\in\RR\backslash\{0\}$, $\veps\in (0,1)$ and $f\in B_p^{s,\infty}(\RR^d)$. Then 
 $$E(\beta,f)\subset\left\{x\in\RR^d;\ \liminf_{j\to+\infty} \sup_{l\in[(1-\veps)j,(1+\veps)j]}\frac{\log |Q_lf(x)|}{j\log 2}\geq\beta\right\}.$$
\end{lemma}
\begin{proof}
 We first assume $\beta>0$. Let $\delta>0$ and pick $x\in E(\beta,f)$. Then, provided $j$ is large enough, we have simultaneously
 \begin{eqnarray*}
  |P_{\lfloor j(1+\veps)\rfloor}f(x)|&\geq& 2^{(1+\veps)(\beta-\delta)j}\\
|P_{\lfloor j(1-\veps)\rfloor}f(x)|&\leq& 2^{(1-\veps)(\beta+\delta)j}.
 \end{eqnarray*}
Hence, for $j$ large enough,
$$\left|P_{\lfloor j(1+\veps)\rfloor}f(x)-P_{\lfloor j(1-\veps)\rfloor}f(x)\right|\geq \frac{1}2 2^{(1+\veps)(\beta-\delta)j}\geq 2^{\beta j}$$
provided
$$(1+\veps)(\beta-\delta)>(1-\veps)(\beta+\delta)\textrm{ and }(1+\veps)(\beta-\delta)>\beta.$$
Both conditions are satisfied if $\delta$ is sufficiently close to $0$. Since
$$\sup_{l\in[(1-\veps)j,(1+\veps)j]} |Q_lf(x)|\geq\frac{1}{2\veps j+1}\left|P_{\lfloor j(1+\veps)\rfloor}f(x)-P_{\lfloor j(1-\veps)\rfloor}f(x)\right|$$
we get the conclusion. The proof for $\beta<0$ is similar, but working now with $R_j$ instead of $P_j$. Indeed, provided $j$ is large enough,
we have simultaneously 
\begin{eqnarray*}
  |R_{\lfloor j(1+\veps)\rfloor}f(x)|&\leq& 2^{(1+\veps)(\beta+\delta)j}\\
|R_{\lfloor j(1-\veps)\rfloor}f(x)|&\geq& 2^{(1-\veps)(\beta-\delta)j}
 \end{eqnarray*}
 and we choose $\delta>0$ such that
$$(1+\veps)(\beta+\delta)<(1-\veps)(\beta-\delta)\textrm{ and }(1-\veps)(\beta-\delta)>\beta.$$
\end{proof}

The next lemma is crucial. It essentially says that if $|Q_lf(x)|$ is large, then the localized norm $\|f\|_{l,\lambda_j(x),\veps}$ is also large.
\begin{lemma}\label{lem:packing1}
 Let $f\in B_p^{s,\infty}(\RR^d)$, $x\in\RR^d$, $\kappa\in\RR$, $\veps\in (0,1)$ and $j,l\in\mathbb N$ with $l\in [(1-\veps)j,(1+\veps)j]$.  Then 
 $$|Q_l f(x)|\leq C 2^{\left(\frac dp-s+\theta(\veps)\right)j} \|f\|_{l,\lambda_j(x),\veps}+C2^{\kappa j}$$
 where $\theta:(0,+\infty)\to(0,+\infty)$ satisfies $\lim_{0^+}\theta=0$.
\end{lemma}
\begin{proof}
 We write
 \begin{eqnarray*}
  |Q_l f(x)|&\leq&\sum_i\sum_{\lambda\in\Lambda_{l,\lambda_j(x),\veps}}|c_\lambda^{(i)}|\cdot \|\psi_\lambda^{(i)}\|_\infty +\sum_i\sum_{\lambda\notin\Lambda_{l,\lambda_j(x),\veps}}|c_\lambda^{(i)}|\cdot |\psi_\lambda(x)|.
 \end{eqnarray*}
 We deduce from Lemma \ref{lem:wavelet2}, \eqref{eq:packing1} and the inequality $|c_\lambda^{(i)}|\leq C 2^{\left(\frac dp-s\right)l}$
  that the last term is majorized by $C 2^{\kappa j}$.
Therefore, H\"older's inequality yields
\begin{eqnarray*}
|Q_l f(x)|&\leq& \left(\sum_i \sum_{\lambda\in \Lambda_{l,\lambda_j(x),\veps} }|c_\lambda^{(i)}|^p\right)^{1/p}\left(\sum_{i}\sum_{\lambda\in \Lambda_{l,\lambda_j(x),\veps}}1\right)^{1/p^*}+C2^{\kappa j}\\
&\leq&C 2^{-\left(s-\frac dp\right)l}\left(2^{{\veps dl}}+2^{\veps dj}\right)^{1/p^*} \|f\|_{l,\lambda_j(x),\veps}+C2^{\kappa j}.
\end{eqnarray*}
Taking into account that $|j-l|\leq\veps j$, we get the result.
\end{proof}
\begin{proof}[Proof of Proposition \ref{prop:packing1}]
Let us fix $\beta\in\left(-s,\frac dp-s\right]\backslash\{0\}$ (the statement is trivial for $\beta=-s$). Let $\gamma\in(-s,\beta)$ and $\veps\in(0,1)$. Then 
 $$ E(\beta,f)\subset\mathcal G^+(\gamma,\veps,f):=\left\{x\in\RR^d;\ \exists J\in\NN,\ \forall j\geq J,\ \sup_{l\in[(1-\veps)j,(1+\veps)j]}|Q_lf(x)|\geq 2^{\gamma j}\right\}.$$
Set $\mathcal G_J^+(\gamma,\veps,f):=\left\{x\in\RR^d;\ \forall j\geq J,\  \sup_{l\in[(1-\veps)j,(1+\veps)j]}|Q_lf(x)|\geq  2^{\gamma j}\right\}.$
We intend to show that, for each $J\geq 1$, $\dboxsup \left(\mathcal G_J^+(\gamma,\veps,f)\right)\leq d-sp-\gamma p+\omega(\veps)$ where $\lim_0^+\omega=0$.
Since $\mathcal G^+(\gamma,\veps,f)\subset\bigcup_J \mathcal G_J^+(\gamma,\veps,f)$, it will follow from \cite[Section 3.3 and 3.4]{Fal03}) that 
$\dimp(E(\beta,f))\leq d-sp-\gamma p+\omega(\veps)$. Letting $\gamma$ to $\beta$ and
 $\veps$ to 0 will then yield the result.

Let $j\geq J$ be large and let $\Theta_j$ be the dyadic cubes of the $j$-th generation intersecting $\mathcal G_J^+(\gamma,\veps,f)$. Let $N_j$ be the cardinal number of $\Theta_j$.
Then, for any $\lambda_0\in\Theta_j$, Lemma \ref{lem:packing1} applied with $\kappa<\gamma$ to some $x\in \lambda_0\cap \mathcal G_J^+(\gamma,\veps,f)$ 
implies that there exists $l_0\in [(1-\veps)j,(1+\veps)j]$ with
$$2^{\gamma pj}\leq |Q_{l_0}f(x)|^p\leq C2^{(d-sp+\theta(\veps)p)j}\sup_{l\in [(1-\veps)j,(1+\veps)j]}\|f\|_{l,\lambda_0,\veps}^p+\frac 12 2^{\gamma p j}.$$
Summing this over all $\lambda_0\in\Theta_j$ we get
\begin{eqnarray*}
 N_j2^{\gamma pj}&\leq&C 2^{(d-sp+\theta(\veps)p)j}\sum_{\lambda_0\in\Theta_j}\sup_{l\in [(1-\veps)j,(1+\veps)j]}\|f\|^p_{l,\lambda_0,\veps}\\
 &\leq&C 2^{(d-sp+\theta(\veps)p)j}\sum_{l\in [(1-\veps)j,(1+\veps)j]}\sum_{\lambda_0\in\Lambda_j}\|f\|^p_{l,\lambda_0,\veps}.
 \end{eqnarray*}
 Now since any $\lambda\in\Lambda_l$ with $|l-j|\leq\veps j$ belongs to at most $C_d 2^{\veps dj+\veps(1+\veps)dj}$ different sets $\Lambda_{l,\lambda_0,\veps}$ for $\lambda_0$ describing $\Lambda_j$,
 we have that for any such $l$
 \[\sum_{\lambda_0\in\Lambda_j} \|f\|_{l,\lambda_0,\veps}^p\leq C 2^{\veps dpj+\veps(1+\veps)dpj}\|f\|_l.\]
 This in turn implies
 \begin{align*}
 N_j2^{\gamma pj}&\leq C 2^{(d-sp+\theta(\veps)p)j}2^{\veps dpj+\veps(1+\veps)dpj}\sum_{l\in [(1-\veps)j,(1+\veps)j]}\|f\|_l^p\\
 &\leq (2\veps j+1)2^{(d-sp+\theta(\veps)p)j}2^{\veps dpj+\veps(1+\veps)dpj} \|f\|^p.
\end{align*}
Thus,
$$\limsup_{j\to+\infty}\frac{\log N_j}{j\log 2}\leq d-sp-\gamma p+\theta(\veps)p+\veps dp+\veps(1+\veps)d,$$
which allows us to conclude.
\end{proof}

\section{Existence of multifractal functions}\label{sec:existence}

Throughout this section, we assume that the wavelets have compact support.
\subsection{A Cantor set with prescribed behaviour of the wavelets}\label{sec:compactset}
The nonpositivity of the wavelets (more precisely, the nonpositivity of $P_j$) add substantial difficulties to the construction of a saturating function $f$ such that $P_j f(x)$ is large for all $j$ and all $x$ in a big set. Our strategy is to force positivity by the construction of a big Cantor set
where we control the behaviour of many $\psi_\lambda^{(1)}$.

\begin{proposition}\label{prop:bigcantorset}
 Let $d'\in (0,d)$. There exist an autosimilar and compact set $K\subset\RR^d$ satisfying the open set condition and two integers $t,N$ such that
 \begin{itemize}
  \item $\dimh(K)=\dimp(K)\geq d'$.
  \item $K$ is the decreasing intersection of compact sets $K_n$, where each $K_n$ is the union of closed dyadic cubes of width $2^{-(t+Nn)}$. We denote
  by $\Theta_n$ the set of closed dyadic cubes of width $2^{-(t+Nn)}$ such that $K_n=\bigcup_{\lambda\in\Theta_n}\lambda$.
  \item To each $\lambda\in \Theta_n$, we may associate a closed dyadic cube $\mu(\lambda)$ of width $2^{-Nn}$ such that, if $\lambda\neq\lambda'\in\Theta_n$, then 
  $\mu(\lambda)\neq \mu(\lambda')$.
  \item For all $x\in K_n$ and all $\lambda\in \Theta_n$, 
  $$
  \begin{cases}
   \psi_{\mu(\lambda)}^{(1)}(x)\geq 1&\textrm{if }x\in \lambda\\
   \psi_{\mu(\lambda)}^{(1)}(x)=0&\textrm{otherwise.}
  \end{cases}$$
 \end{itemize}
\end{proposition}
\begin{proof}
 To simplify the notations, we will only provide a proof for the one-dimensional case.
 Rescaling $\psi=\psi^{(1)}$ if necessary, we may assume that $\psi\geq 1$ on some dyadic interval $\left[\frac k{2^t},\frac{k+1}{2^t}\right]$
 and that $\psi=0$ outside $[0,1]$. Let $N\geq t$ be a very large integer and 
 $$\Omega=\left\{2^{N-t}k,2^{N-t}k+1,\dots,2^{N-t}k+2^{N-t}-1\right\}.$$
 For $m\in\Omega$, let $s_m$ be the similarity $s_m(x)=\frac{1}{2^N}x+\frac{m}{2^N}$. We start from $K_0=\left[\frac k{2^t},\frac{k+1}{2^t}\right]$
 and we observe that the choice of $\Omega$ is done in order to ensure that $s_m(K_0)\subset K_0$ for all $m\in\Omega$. 
 Define inductively $K_n=\bigcup_{m\in\Omega}s_m(K_{n-1})$ and $K=\bigcap_{n\geq 0}K_n$. The compact set $K$ satisfies
 the open set condition, namely there exists a nonempty bounded open set $V$ such that $V\supset \bigcup_{m\in\Omega}s_m(V)$ where the union is disjoint.
 For instance, the set $\left(\frac k{2^t},\frac{k+1}{2^t}\right)$ does the job since
 $$s_m\left(\left(\frac k{2^t},\frac{k+1}{2^t}\right)\right)=\left(\frac{k+m2^t}{2^{N+t}},\frac{k+m2^t+1}{2^{N+t}}\right).$$
 It follows from the standard theory of autosimilar sets (see e.g. \cite{Fal03}) that
 $\dimh(K)=\dimp(K)=\kappa$ where $\kappa$ is the solution of 
 $$\textrm{card}(\Omega)\times\frac{1}{2^{N\kappa}}=1\iff \frac{2^{N-t}}{2^{N\kappa}}=1.$$
 Letting $N$ to infinity, we may be sure that $\kappa$ is as close to 1 as we want.
 
 Each $K_n$ consists of closed dyadic intervals of width $2^{-(t+Nn)}$. We denote by $\Theta_n$ the set of these intervals.
 We prove by induction on $n$ that any $\lambda\in\Theta_n$
 can be written (uniquely) $\lambda=\left[\frac{k+l2^t}{2^{t+Nn}},\frac{k+l2^t+1}{2^{t+Nn}}\right]$. 
 This is true for $n=0$. If we assume that this is true up to $n$, then any $\lambda\in\Theta_{n+1}$ is equal to 
 \begin{align*}
  \lambda&=s_m\left(\left[\frac{k+l2^t}{2^{t+Nn}},\frac{k+l2^t+1}{2^{t+Nn}}\right]\right)\\
  &=\left[\frac{k+2^t(l+2^{Nn}m)}{2^{t+N(n+1)}},\frac{k+2^t(l+2^{Nn}m)+1}{2^{t+N(n+1)}}\right]
 \end{align*}
 for some $l,m$. 
 
 We then define
 $\mu(\lambda)=\left[\frac{l}{2^{Nn}},\frac{l+1}{2^{Nn}}\right]$ so that, for $\lambda\neq\lambda'\in\Theta_n$,
 we indeed have $\mu(\lambda)\neq\mu(\lambda')$. Finally, if $x$ belongs to $\lambda=\left[\frac{k+l2^t}{2^{t+Nn}},\frac{k+l2^t+1}{2^{t+Nn}}\right]$,
 then $\psi_{\mu(\lambda)}(x)=\psi(2^{Nn}x-l)$ and it is easy to check that $2^{Nn}x-l\in\left[\frac k{2^t},\frac{k+1}{2^t}\right]$
 so that $\psi_{\mu(\lambda)}(x)\geq 1$. On the other hand, if $x\in \lambda'=\left[\frac{k+l'2^t}{2^{t+Nn}},\frac{k+l'2^t+1}{2^{t+Nn}}\right]$
 with $\lambda\neq \lambda'$, then $2^{Nn}x-l\notin [0,1]$, so that $\psi_{\mu(\lambda)}(x)=0$.
\end{proof}

\subsection{The saturating functions - case of divergence}

To prove part (ii) and (iii) of Theorem \ref{thm:main1}, we begin with the construction of one function whose wavelet series diverges fast on a set with given upper box dimension and which is moreover
nonnegative. We also assume that $s>0$ so that the interval $(d-sp,d)$ is not empty. 

\begin{theorem}\label{thm:saturatingdiv1}
Let $d'\in (d-sp,d)$ and let $K$ be given by Proposition \ref{prop:bigcantorset}. For all $\alpha\in(0,d-sp)$,
for all $G\subset K$ with $\dboxsup(G)<\alpha$, there exists $f\in B_p^{s,1}(\RR^d)$, $\|f\|\leq 1$ such that
\begin{itemize}
\item for all $x\in K$, for all $j\in\NN$, $P_j f(x)\geq 0$;
\item for all $x\in G$, $\liminf_j \frac{\log |P_j f(x)|}{j\log 2}\geq\frac{d-sp-\alpha}{p}$.
\end{itemize}
\end{theorem}
\begin{proof}
 Let $\alpha'\in (\dboxsup(G),\alpha)$. Let $\Gamma_n\subset\Theta_n$ be the dyadic balls of width $2^{-(t+Nn)}$ intersecting $G$. One knows that $\textrm{card}(\Gamma_n)\leq C_G 2^{(t+Nn)\alpha'}$. 
 Define
 $$f_n=2^{(Nn+t)\times\frac{d-sp-\alpha'}p}\sum_{\lambda\in \Gamma_n}\psi_{\mu(\lambda)}^{(1)}$$
 so that, since for $\lambda\in\Gamma_n$, $\mu(\lambda)$ is a cube of the $Nn$-th generation,
 \begin{eqnarray*}
  \|f_n\|&\leq &2^{(Nn+t)\times\frac{d-sp-\alpha'}p} \left(\textrm{card}(\Gamma_n)\right)^{1/p} 2^{Nn\times\frac{sp-d}p}\\
  &\leq&C_{G}.
 \end{eqnarray*}
 We then set $f=\sum_{n\geq 1}n^{-2}f_n$. For $x\in K$ and $j\in (Nn,N(n+1)]$, 
 $$P_jf(x)=\sum_{l\leq n}l^{-2}2^{(Nl+t)\times\frac{d-sp-\alpha'}{p}}\sum_{\lambda\in \Gamma_l} \psi_{\mu(\lambda)}^{(1)}(x)$$
 and this is always nonnegative. Moreover, if $x$ belongs to $G$, then $x$ belongs to some $\lambda\in\Gamma_{n}$ so that 
 $$P_j f(x)\geq n^{-2}2^{\big(Nn+t\big)\times\frac{d-sp-\alpha'}{p}}.$$
 This shows that 
 $$\liminf_j \frac{\log P_j f(x)}{j\log 2}\geq\frac{d-sp-\alpha'}p\geq\frac{d-sp-\alpha}p.$$
\end{proof}

For our proof of prevalence, we will need a variant of the previous result.

\begin{theorem}\label{thm:saturatingdiv2}
Let $d'\in (d-sp,d)$ and let $K,N,t$ be given by Proposition \ref{prop:bigcantorset}. Let also $J\geq 1$. For all $\alpha\in(0,d-sp)$,
for all $G\subset K$ with $\dboxsup(G)<\alpha$, there exist $f_0,\dots,f_{J-1}\in B_p^{s,1}(\RR^d)$, $\|f_k\|\leq 1$ such that
\begin{itemize}
\item for all $x\in K$, for all $j\in\NN$, for all $k\in\{0,\dots,J-1\}$, $P_j f_k(x)\geq 0$;
\item for all $x\in G$, for all $j\in\NN$, for all $k\in\{0,\dots,J-1\}$, for all $l\in\{0,\dots,JN-1\}$ with $kN\neq l$, $Q_{jJN+l}f_k(x)=0$;
\item for all $k\in\{0,\dots,J-1\}$, $\liminf_j \inf_{x\in G} \frac{\log |Q_{jJN+kN} f_k(x)|}{jJN\log 2}\geq\frac{d-sp-\alpha}{p}$.
\end{itemize}
\end{theorem}
\begin{proof}
 The proof is identical to that of Theorem \ref{thm:saturatingdiv1} except that we now set 
 \[f_k = \sum_{\substack{n\geq 1;\\ n=k\ [J]}}n^{-2}f_n.\]
\end{proof}

\subsection{Existence of one strongly multifractal functions - the divergence case}\label{sec:prevalence}

We now go from the existence of one function $f$ with control of $P_j f(x)$ on a set of given upper box dimension
to the existence of a function $f$ with control of $P_j f(x)$ on a set of given packing dimension.
We recall that $X=B_p^{s,q}(\RR^d)$ or that $X=W^{p,s}(\RR^d)$ and we still assume that $s>0$.

\begin{lemma}\label{lem:strongly1}
Let $d'\in (d-sp,d)$ and let $K$ be given by Proposition \ref{prop:bigcantorset}. For all $\alpha\in(0,d-sp)$,
for all $F\subset K$ with $\dimp(F)=\alpha$, there exists $f\in X$, $\|f\|\leq 1$ such that
\begin{itemize}
\item for all $x\in K$, for all $j\in\NN$, $P_j f(x)\geq 0$;
\item for all $x\in F$, $\liminf_j \frac{\log |P_j f(x)|}{j\log 2}\geq\frac{d-sp-\alpha}{p}$.
\end{itemize}
\end{lemma}
\begin{proof}
 Let $(\alpha_l)$ be a sequence decreasing to $\alpha$. Then there exists a sequence $(G_{l,u})$ of subsets of $(0,1)^d$ such that $F\subset\bigcap_l \bigcup_u G_{l,u}$ and $\dboxsup(G_{l,u})<\alpha_l$.
 We apply Theorem \ref{thm:saturatingdiv1} with $G=G_{l,u}\cap K$ and $\alpha=\alpha_l$ to get a function $f_{l,u}$ and we set $f=\sum_{l,u}2^{-(l+u)}f_{l,u}$. 
 Then, for any $x\in K$, 
 $$P_j f(x)\geq\sum_{l,u}2^{-(l+u)}P_j f_{l,u}(x)\geq 0.$$
 Moreover, let $l\geq 1$ and $x\in F$. There exists $u$ such that $x\in G_{l,u}$. Then from $P_j f(x)\geq 2^{-(l+u)}P_j f_{l,u}(x)$ we deduce that
 \[ \liminf_j \frac{\log |P_j f(x)|}{j\log 2}\geq \frac{d-sp-\alpha_l}{p}. \]
 Since $\alpha_l$ may be taken arbitrarily close to $\alpha$, we get the result.
\end{proof}

We are now ready for the proof of part (iii) of Theorem \ref{thm:main1}.
\begin{proof}[Proof of part (iii) of Theorem \ref{thm:main1}]
 Let $d'\in(d-sp,d)$ and let $K$ be the compact set given by Proposition \ref{prop:bigcantorset}.
Since $K$ is an autosimilar and compact set satisfying the open set condition with $\dimh(K)=d'>d-sp$, there exists $(F_\beta)_{\beta\in \left(0,\frac dp-s\right)}$, a decreasing family of compact 
subsets of $K$, such that, for all $\beta\in \left(0,\frac dp-s\right)$, 
$$\dimh(F_\beta)=\dimp(F_\beta)=d-sp-\beta p\textrm{ and }\mathcal H^{d-sp-\beta p}(F_\beta)>0$$
(see \cite{BAYMFF}).
Let $(\beta_k)$ be a dense sequence in  $\left(0,\frac dp-s\right)$. For any $k\geq 1$, Lemma \ref{lem:strongly1} yields the existence of a function $f_k\in X$, $\|f_k\|\leq 1$, such that,
for all $x\in F_{\beta_k}$, 
$$\liminf_j \frac{\log |P_j f_k(x)|}{j\log 2}\geq\beta_k$$
and $P_j f_k(y)\geq 0$ for all $y\in K$. We set $f=\sum_k 2^{-k}f_k$ and let $\beta\in  \left(0,\frac dp-s\right)$. 
Taking $(\beta_{\phi(k)})$ a subsequence of $(\beta_k)$ increasing to $\beta$, we get for all $k$ and all $x\in F_\beta\subset F_{\beta_{\phi(k)}}$,
$$\liminf_j \frac{\log |P_j f(x)|}{j\log 2}\geq\beta_{\phi(k)}.$$
Hence,
$$\liminf_j \frac{\log |P_j f(x)|}{j\log 2}\geq\beta.$$
Now, we can decompose $F_\beta$  into
$$F_\beta= \big(F_\beta\cap E(\beta,f)\big)\cup\bigcup_{\gamma>\beta} \big(F_\beta\cap E^-(\gamma,f)\big).$$
Because of Proposition \ref{prop:dimh}, $\mathcal H^{d-sp-\beta p}\big(E^-(\gamma,f)\big)=0$ for all $\gamma>\beta$ so that $\dimh\big(F_\beta\cap E(\beta,f)\big)=d-sp-\beta p$. This yields the conclusion, since
$$d-sp-\beta p\leq \dimh\big(F_\beta\cap E(\beta,f)\big)\leq \dimh\big(E(\beta,f)\big)
\leq \dimp\big(E(\beta,f)\big)\leq d-sp-\beta p.$$
Observe also that, because of part (i), there is nothing to do for $\beta=\frac dp-s$.
\end{proof}

\subsection{Residuality of multifractal functions - the divergence case}

We intend to prove the residual part of Theorem \ref{thm:main1} (ii). 
Again we assume $s>0$ and pick $d'\in(d-sp,d)$. Let us fix $K$ the compact set given by Proposition \ref{prop:bigcantorset}. 
 Our first step is to exhibit, for all compact sets $F\subset K$ with $\dimh(F)=\alpha$, 
 a residual set $\mathcal R_F$ such that, for all $f\in\mathcal R_F$, for all $x\in F$, 
$$\limsup_j \frac{\log |P_j f(x)|}{j\log 2}\geq\frac{d-sp-\alpha}p.$$
Although this construction can be carried on for all such subsets $F$ (and even without the restriction $F\subset K$),
we will impose that $\dimp(F)=\alpha$. In that case, the construction is simplified by the existence of one function satisfying the stronger property:
\begin{equation}\label{eq:residuality1}
\forall x\in F,\ \liminf_j \frac{\log |P_j f(x)|}{j\log 2}\geq\frac{d-sp-\alpha}p.
\end{equation}

\begin{lemma}\label{lem:residualfixedlevel}
Let $d'\in (d-sp,d)$ and let $K$ be given by Proposition \ref{prop:bigcantorset}. For all $\alpha\in(0,d-sp)$,
for all compact sets $F\subset K$ with $\dimp(F)=\alpha$, there exists a residual subset $\mathcal R_F\subset X$ such that, for all $f\in\mathcal R_F$, for all $x\in F$, 
$$\limsup_j \frac{\log |P_j f(x)|}{j\log 2}\geq\frac{d-sp-\alpha}p.$$
\end{lemma}
\begin{proof}
Let $f\in X$ satisfying \eqref{eq:residuality1}. Let $(f_l)$ be a dense sequence in $X$ such that $f_l\in V_l$ for all $l$
(recall that we assume $p,q\neq\infty$). Let finally $(\alpha_l)$ be a sequence decreasing to $\alpha$. We define $g_l=f_l+\frac 1l f$. Then, for all $l\geq 1$ and all $x\in F$, there exists an integer $J_{l,x}$ such that, for all $j\geq J_{l,x}$, 
$$|P_j g_l(x)|>2^{\frac{d-sp-\alpha_l}p j}.$$
By compactness of $F$, $J_l:=\max\left\{J_{l,x};\ x\in F\right\}$ does exist. Let now $m\geq 1$ and set $j_{l,m}=\max(J_l,m)$.  There exists $\delta_{l,m}>0$ such that, for all $g\in B_X(g_l,\delta_{l,m})$, for all $x\in F$, 
$$|P_{j_{l,m}}g(x)|\geq  2^{\frac{d-sp-\alpha_l}p j_{l,m}}.$$
Define $\mathcal R_F=\bigcap_{m\geq 1}\bigcup_{l\geq m}B_X(g_l,\delta_{l,m})$ which is a residual subset of $X$. Pick $g\in\mathcal R_F$ and $m\geq 1$.
There exists $l\geq m$ such that $g\in B_X(g_l,\delta_{l,m})$. Then there exists $j\geq m$ such that, for all  $x\in F$, 
$$|P_{j}g(x)|\geq  2^{\frac{d-sp-\alpha_l}p j}\geq 2^{\frac {d-sp-\alpha_m}p j}.$$
Since $(\alpha_m)$ goes to $\alpha$, we are done.
\end{proof}
\begin{remark}
The functions in $\mathcal R_F$ have a stronger property. Indeed, they satisfy
$$\limsup_j \inf_{x\in F}\frac{\log |P_j f(x)|}{j\log 2}\geq \frac{d-sp-\alpha}{p}.$$
\end{remark}

We are now ready to prove the residual part of Theorem \ref{thm:main1} (ii).

\begin{proof}[Proof of Theorem \ref{thm:main1}, part (ii), residuality]
We start as in subsection \ref{sec:prevalence} by fixing $d'\in(d-sp,d)$, $K$ the compact set given by Proposition \ref{prop:bigcantorset}
and $(F_\beta)_{\beta\in \left(0,\frac dp-s\right)}$ a decreasing family of compact 
subsets of $K$, such that, for all $\beta\in \left(0,\frac dp-s\right)$, 
$$\dimh(F_\beta)=\dimp(F_\beta)=d-sp-\beta p\textrm{ and }\mathcal H^{d-sp-\beta p}(F_\beta)>0.$$
Let also $(\beta_k)$ be a dense sequence in  $\left(0,\frac dp-s\right)$. For any $k\geq 1$, Lemma \ref{lem:residualfixedlevel} yields the existence of a residual set $Y_k$
such that, for all $f\in Y_k$, for all $x\in F_{\beta_k}$, 
$$\limsup_j \frac{\log |P_j f(x)|}{j\log 2}\geq\beta_k.$$
Set $Y=\bigcap_{k\geq 1}Y_k$ (which remains residual) and let $\beta\in  \left(0,\frac dp-s\right)$, $f\in Y$.
Taking $(\beta_{\phi(k)})$ a subsequence of $(\beta_k)$ increasing to $\beta$, we get for all $k$ and all $x\in F_\beta\subset F_{\beta_{\phi(k)}}$,
$$\limsup_j \frac{\log |P_j f(x)|}{j\log 2}\geq\beta_{\phi(k)}.$$
Hence,
$$\limsup_j \frac{\log |P_j f(x)|}{j\log 2}\geq\beta.$$
Now, we can decompose $F_\beta$  into
$$F_\beta= \big(F_\beta\cap E^-(\beta,f)\big)\cup\bigcup_{\gamma>\beta} \big(F_\beta\cap E^-(\gamma,f)\big).$$
Again an application of Proposition \ref{prop:dimh} yields $\dimh\big(F_\beta\cap E(\beta,f)\big)=d-sp-\beta p$. This in turn implies the conclusion, 
since
$$d-sp-\beta p\leq \dimh\big(F_\beta\cap E^-(\beta,f)\big)\leq \dimh\big(E^-(\beta,f)\big)\leq  d-sp-\beta p.$$
\end{proof}

\subsection{Prevalence of multifractal functions - the divergence case}

We need a substitute to Lemma \ref{lem:residualfixedlevel} where we replace residuality by prevalence.

\begin{lemma}\label{lem:prevalencesaturatingdiv}
Let $d'\in(d-sp,d)$ and let $K$ be the compact set given by Proposition \ref{prop:bigcantorset}.
For all $\alpha\in(0,d-sp)$, for all $F\subset K$ with $\dimp(F)=\alpha$,  there exists a prevalent set $Y_F\subset X$ such that, for all $f\in Y_F$, for all $x\in F$, 
$$\limsup_j \frac{\log |P_j f(x)|}{j\log 2}\geq \frac{d-sp-\alpha}p.$$
\end{lemma}
\begin{proof}
Let us set, for $0<\beta\leq (d-sp-\alpha)/p$ and $L$ a subset of $(0,1)^d$, 
\[Y_{\beta,L}=\left\{f\in X;\ \forall x\in L,\ \limsup_j \frac{\log |P_j f(x)|}{j\log 2}\geq\beta\right\}.\]
Since $Y_{(d-sp-\alpha)/p,F}=\bigcap_{\beta<(d-sp-\alpha)/p}Y_{\beta,F}$, we just need to prove that each $Y_{\beta,F}$ is prevalent, for $\beta<(d-sp-\alpha)/p$.
We fix such a $\beta$.
Lemma \ref{lem:wavelet3} tells us that there exists $\theta>0$ such that, for all $j\in \NN$, all $x,y\in\RR^d$ with $\|x-y\|\leq 2^{-\theta j}$, all $f\in X$,
$$|Q_j f(x)-Q_j f(y)|\leq  C_\beta 2^{\beta j}\|f\|_X.$$

We also consider $\gamma\in(\beta,(d-sp-\alpha)/p)$. 
There exists a sequence $(G_u)$ of subsets of $K$ such that $F\subset \bigcup_u G_u$
and $\dboxsup(G_u)<(d-sp-\gamma p)/p$.

Observing that $Y_{\beta,F}\supset\bigcap_u Y_{\beta,G_u}$, we just need to prove that each $Y_{\beta,G}$ is prevalent, where $G\subset K$ is such that $\dboxsup(G)<(d-sp-\gamma p)/p$. 
We fix such a set $G$. For each $j\geq 1$, $G$ is the union of at most $2^{\kappa j}$ cubes of width $2^{-\theta j}$ where the exact value
of $\kappa$ is unimportant for us. These cubes will be denoted by $O_{j,l}$.

Let $J\geq 1$ be any integer satisfying $(J-1)(\gamma-\beta)>\kappa$. We apply Theorem \ref{thm:saturatingdiv2} with these values of $G$ and $J$ to get functions $f_0,\dots,f_{J-1}$. 
In particular, for all $j$ large enough, for all $x\in G$, for all $k\in\{1,\dots,J-1\}$, 
$$|Q_{jJN+kN}f_k(x)|\geq 2^{\gamma jJN}$$
(recall that $N$ is a fixed integer which is defined during the construction of $K$).
We fix $f\in X$ and set, for $j\geq 1$,
\begin{align*}
 U_{j,l}&=\Big\{(c_1,\dots,c_{J-1})\in [0,1]^{J-1};\ \exists x\in O_{jJN,l},\ \forall m\in\{0,\dots,JN-1\},\\
 &\quad\quad\quad\quad\quad |P_{jJN+m}(f+c_1 f_1+\dots+c_{J-1}f_{J-1})(x)|\leq 2^{\beta jJN}\Big\}\\
 U_j&=\bigcup_l U_{j,l}.
\end{align*}
We claim that $\lambda_{J-1}(\liminf_j U_j)=0$ where $\lambda_{J-1}$ is the Lebesgue measure on $\RR^{J-1}$. Accept this claim for a while. Then, for almost all $(c_1,\dots,c_{J-1})\in [0,1]^{J-1}$, 
$f+c_1 f_1+\cdots+c_{J-1}f_{J-1}$ belongs to $\limsup_j U_j^c$. In particular, for all $j_0\in\NN$, there exists $j\geq j_0$ such that, for 
all $l$ and all $x\in O_{jJN,l}$, there exists $m\in\{0,\dots,JN-1\}$ with
\[ |P_{jJN+m}(f+c_1 f_1+\dots+c_{J-1}f_{J-1})(x)|\geq 2^{\beta jJN}. \]
Since $G\subset\bigcup_l O_{jJN,l}$, this implies that for infinitely many $j$, for all $x\in G$, 
there exists $m\in\{0,\dots,JN-1\}$ with
\[ |P_{jJN+m}(f+c_1 f_1+\dots+c_{J-1}f_{J-1})(x)|\geq 2^{\beta jJN}. \]
In other words, $f+c_1 f_1+\cdots+f_{J-1} f_{J-1}$ belongs to $Y_{\beta,G}$ for almost all $(c_1,\dots,c_{J-1})\in [0,1]^{J-1}$,
proving that $Y_{\beta,G}$ is prevalent. 

Thus, it remains to prove the claim. We just need to prove that $\lambda_{J-1}(U_j)\leq C2^{-\omega j}$ for some $C,\omega>0$.
Since $\lambda_{J-1}(U_j)\leq 2^{\kappa jJN}\max_l \lambda_{J-1}(U_{j,l})$, it remains to show that $\lambda_{J-1}(U_{j,l})\leq 2^{-(\kappa+\omega)jJN}$ for all $j$ and $l$. 
Let $(c_0,\dots,c_{J-1})$ and $(d_0,\dots,d_{J-1})$ belonging to $U_{j,l}$. There exist $x,y\in O_{jJN,l}$ such that, for all $m\in\{0,\dots,JN-1\}$, 
\begin{align*}
 |P_{jJN+m}(f+c_1 f_1+\dots+c_{J-1}f_{J-1})(x)|&\leq 2^{\beta jJN} \\
 |P_{jJN+m}(f+d_1 f_1+\dots+d_{J-1}f_{J-1})(y)|&\leq 2^{\beta jJN}.
\end{align*}
We look at these two inequalities for $m=kN$ and $m=kN+1$, $k=1,\dots,J-1$. Taking the difference and using the triangle inequality, we get
\begin{align*}
|Q_{jJN+kN} f(x)+c_k Q_{jJN+kN}f_k(x)|&\leq 2\cdot 2^{\beta jJN}\\
|Q_{jJN+kN} f(y)+d_k Q_{jJN+kN}f_k(y)|&\leq 2\cdot 2^{\beta jJN}.
\end{align*}
Using another time the triangle inequality and writing
\begin{align*}
c_k Q_{jJN+kN}f_k(x)-d_k Q_{jJN+kN}f_k(y)&=(c_k-d_k)Q_{jJN+kN}f_k(x)+\\
&\quad d_k\big(Q_{jJN+kN} f_k(x)-Q_{jJN+kN}f_k(y)\big), 
\end{align*}
we get 
\begin{align*}
|c_k-d_k| \cdot |Q_{jJN+kN}f_k(x)|&\leq 4\cdot 2^{\beta jJ}+\big|Q_{jJN+kN}f(x)-Q_{jJN+kN}f(y)\big|+\\
&\quad\quad |d_k|\cdot\big|Q_{jJN+kN}f_k(x)-Q_{jJN+kN}f_k(y)\big|\\
&\leq C2^{\beta jJN}
\end{align*}
since $\|x-y\|\leq 2^{-\theta jJN}$.
Hence, $|c_k-d_k|\leq C2^{(\beta-\gamma)jJN}$. Therefore, the set of $c=(c_1,\dots,c_{J-1})\in[0,1]^{J-1}$ belonging to $U_{j,l}$ is contained in a cube of width $C2^{(\beta-\gamma)jJN}$.
Hence, $\lambda_{J-1}(U_{j,l})\leq C2^{(\beta-\gamma)(J-1) J Nj}$. Because $(\gamma-\beta)(J-1)>\kappa$, we are done.
\end{proof}

The proof of the prevalence part of Theorem \ref{thm:main1}, (ii), follows now from an argument similar to that of the residual part. We omit the details.

\begin{remark}\label{rem:prevalenceprecised}
 In fact, our proof of Lemma \ref{lem:prevalencesaturatingdiv} shows a stronger result: 
 since $\lambda_{J-1}(U_j)\leq 2^{-\delta j}$ for some $\delta>0$, we in fact have $\lambda_{J-1}(\liminf_j U_j)=0$.
 In particular, we have shown that, for all $G\subset K$ with $\dboxsup(G)<\alpha$, for all $\beta<(d-sp-\alpha)/p$, there exists $J\geq 1$ such that the set of functions
 $f\in X$ satisfying 
 $$\liminf_{j\to+\infty}\sup_{k=0,\dots,JN-1}\frac{\log |P_{jJN+k}f(x)|}{jJN\log 2}\geq \beta$$
 for all $x\in G$ is prevalent. We are not so far from the proof that the set of functions satisfying (iii) in Theorem \ref{thm:main1} is prevalent. See also Section \ref{sec:remarks}.
\end{remark}

\subsection{The case of convergence}

We now indicate briefly how to modify the previous work to obtain Theorem \ref{thm:main2} for $\beta\neq -s$. 
We only consider the most difficult case (the existence of a prevalent set of multifractal functions) the other cases being left to the reader.
The analogue  of Theorem \ref{thm:saturatingdiv1} reads:

\begin{theorem}\label{thm:saturatingconv1}
Let $d'\in (0,d)$ and let $K$ be given by Proposition \ref{prop:bigcantorset}. For all $\alpha\in(\max(0,d-sp),d')$,
for all $G\subset K$ with $\dboxsup(G)<\alpha$, there exists $f\in B_p^{s,1}(\RR^d)$, $\|f\|\leq 1$ such that
\begin{itemize}
\item for all $x\in K$, $(P_j f(x))$ converges.
\item for all $x\in K$, for all $j\in\NN$, $R_j f(x)\geq 0$;
\item for all $x\in G$, $\liminf_j \frac{\log |R_j f(x)|}{j\log 2}\geq\frac{d-sp-\alpha}{p}$.
\end{itemize}
\end{theorem}
\begin{proof}
 Let $\alpha'\in \big(\max(\dboxsup(G),d-sp),\alpha\big)$. Keeping the notations of the proof of Theorem \ref{thm:saturatingdiv1}, 
 we still set 
 $$f_n=2^{(Nn+t)\times\frac{d-sp-\alpha'}p}\sum_{\lambda\in \Gamma_n}\psi_{\mu(\lambda)}^{(1)}$$
 and $f=\sum_{n\geq 1}n^{-2}f_n$. The convergence of $(P_jf(x))$ for all $x\in\RR^d$ is ensured by the inequality $\alpha'>d-sp$
 (recall that the wavelets have compact support, so that $\psi_{\mu(\lambda)}^{(1)}(x)\neq 0$ for a finite number of $\lambda\in\Gamma_n$,
 this bound being uniform in $n$ and $x$). Moreover, 
for $x\in K$ and $j\in (Nn,N(n+1)]$, 
 $$R_jf(x)=\sum_{l\geq n}l^{-2}2^{(Nl+t)\times\frac{d-sp-\alpha'}{p}}\sum_{\lambda\in \Gamma_l} \psi_{\mu(\lambda)}^{(1)}(x)$$
 and this is always nonegative. Finally, if $x$ belongs to $G$, then $x$ belongs to some $\lambda\in\Gamma_{n}$ so that 
 $$R_j f(x)\geq n^{-2}2^{(Nn+t)\times\frac{d-sp-\alpha'}{p}}.$$
 This shows that 
 $$\liminf_j \frac{\log R_j f(x)}{j\log 2}\geq\frac{d-sp-\alpha'}p\geq\frac{d-sp-\alpha}p.$$
\end{proof}

We then deduce the following lemma.

\begin{lemma}\label{lem:prevalencesaturatingconv}
Let $d'\in (0,d)$, let $K$ be given by Proposition \ref{prop:bigcantorset} and let $\alpha\in\big[\max(0,d-sp),d'\big)\backslash\{0\}$. For all compact subsets $F$ of $K$ with $\dimp(F)=\alpha$, there exists a prevalent set
$Y_F\subset X$ such that, for all $f\in Y_F$, for all $x\in F$, either $(P_jf(x))$ diverges or $(P_jf(x))$ converges and
$$\limsup_j \frac{\log |R_j f(x)|}{j\log 2}\geq \frac{d-sp-\alpha}p.$$
\end{lemma}

The proof of this lemma is almost identical to that of Lemma \ref{lem:prevalencesaturatingdiv}. We now set
\begin{align*}
 Y_{\beta,L}=&\Bigg\{f\in X;\ \textrm{for all }x\in L,  \textrm{ either }(P_jf(x))\textrm{ diverges or }\\
 &\quad (P_jf(x))\textrm{ converges and }\limsup_j \frac{\log |R_jf(x)|}{j\log 2}\geq \beta\Bigg\}.
\end{align*}
The proof that $Y_{(d-sp-\alpha)/p,F}$ is prevalent follows the same lines except that, $G$ and $f$ being fixed,
we define $\hat G=G\cap\{x;\ (P_j f(x)) \textrm{converges}\}$ and we consider the intersection of the cubes $O_{j,l}$ with $\hat G$.
Then, mimicking the proof of Lemma \ref{lem:prevalencesaturatingdiv} (but replacing $P_j$ by $R_j$), we get that, for almost all $(c_1,\dots,c_{J-1})\in [0,1]^{J-1}$,
for all $x\in\hat G$, 
$$\limsup_j \frac{\log|R_j(f+c_1f_1+\cdots+c_{J-1}f_{J-1})(x)|}{j\log 2}\geq \beta.$$
Since for all $x\in G\backslash \hat G$ and all $c_1,\dots,c_{J-1}\in [0,1]^{J-1}$, $P_j(f+c_1 f_1+\cdots+c_{J-1}f_{J-1})(x)$
diverges, we are done.

\medskip

From this lemma, mimicking the work done in Section \ref{sec:prevalence}, we deduce that for all $d'\in (0,d)$, there exists a prevalent subset of functions $Y_{d'}$ such that, for all $f\in Y_{d'}$, for all 
$\beta\in\left(-s+\frac{d-d'}p,\min\left(0,\frac dp-s\right)\right]\backslash\{0\}$,
$$\dimh\big(E(\beta,f)\big)=\dimp\big(E(\beta,f)\big)=d-sp-\beta p.$$
The somehow strange value $-s+\frac{d-d'}p$ comes from the change of variables $\beta=(d-sp-\alpha)/p$
which changes the inequality $\alpha<d'$ to $\beta>-s+(d-d')/p$.
The only important change to do is the decomposition of $F_\beta$. Indeed, it could be possible that, for a given function $f\in Y_{d'}$ and a given $x\in F_\beta$, 
the sequence $(P_jf(x))$ diverges. Setting $\mathcal G=\{x\in \RR^d;\ (P_jf(x))\textrm{ diverges}\}$, we now write
$$F_\beta=\big[F_\beta\cap E(\beta,f)\big]\cup \bigcup_{\gamma\in(\beta,0)}\big[F_\beta\cap E^-(\gamma,f)\big]\cup [F_\beta\cap \mathcal G].$$
We then apply both Proposition \ref{prop:dimh} and \ref{prop:dimhdv} to show that $\dimh\big(F_\beta\cap E(\beta,f)\big)\geq d-sp-\beta p$.

Finally, we get Part (iii) of Theorem \ref{thm:main2} 
for $\beta\in\left(-s,\min\left(0,\frac dp-s\right)\right]\backslash\{0\}$
by setting $Y=\bigcap_{n}Y_{d_n}$, where $(d_n)$ is a sequence increasing to $d$.

\subsection{The case of divergence and $s=0$}

So far, we did not prove Theorem \ref{thm:main1} (ii) and (iii) for $s=0$. Now, we cannot apply Proposition \ref{prop:bigcantorset} with any $d'\in (d-sp,d)$. We may just apply it
for any $d'\in (0,d)$. In the subsequent lemmas (\ref{lem:strongly1}, \ref{lem:residualfixedlevel}, \ref{lem:prevalencesaturatingdiv}), $\alpha$
is now only allowed to go until $d'$. Thus, our proofs just show that, for all $d'\in(0,d)$,
\begin{itemize}
 \item there exists a residual and prevalent subset $Y_{d'}$ of $X$ such that, for all $f\in Y_{d'}$, for all $\beta\in\big((d-d')/p,d/p\big)$, 
 $$\dimh(E^-(\beta,f))=d-\beta p.$$
 \item there exists $f_{d'}\in X$ such that, for all $\beta\in \big((d-d')/p,d/p\big)$, 
 $$\dimh(E(\beta,f))=\dimp(E(\beta,f))=d-\beta p.$$
\end{itemize}
To prove (ii) even if $s=0$, we fix a sequence $(d'_n)$ going to $d$ and we just set $Y=\bigcap_n Y_{d'_n}$. To prove (iii), we observe that, for a fixed $n\in\NN$,
the function $f_{d'_n}$ can be chosen with support in $(n,n+1)$. Then the function $f=\sum_{n\geq 1}n^{-2}f_{d'_n}$ will do the job.

\subsection{The case of convergence and $\beta=-s$}

We also did not prove Theorem \ref{thm:main2} (ii) and (iii) for $\beta=-s$. This was impossible with the method applied before because we constructed our sets $E(\beta,f)$ 
inside a set with packing dimension (strictly) smaller than $d$ whereas we hope to obtain $\dimp\big(E(-s,f)\big)=d$. Therefore,
we will need to enlarge our initial compact set. For simplicity, we again assume $d=1$. We concentrate ourselve on the existence of a strongly multifractal function.
The proof of (ii) will then follow by adapting arguments of the previous subsections.

Recall that a gauge function is a nondecreasing continuous function $\phi:\RR_+\to\RR_+$ satisfying $\phi(0)=0$. The $\phi$-Hausdorff outer measure of a set $E\subset \mathbb R^d$ is 
$$\mathcal H^{\phi}(E)=\lim_{\veps\to 0}\inf_{r\in R_\veps(E)}\sum_{B\in r}\phi(|B|),$$
$R_\veps(E)$ being the set of countable coverings of $E$ with balls $B$ of diameter $|B|\leq\veps$. 
The work done until now 
points out that it is sufficient to find a single function
$f\in B_p^{s,1}(\RR^d)$ satisfying $\mathcal H^\phi\big(\mathcal E^+(-s,f)\big)>0$  for some gauge function $\phi$ such that $\phi(s)=_0 o(s^\gamma)$ for
all $\gamma\in (0,1)$: since $\mathcal H^\phi\big(\mathcal E^{-}(\gamma,f)\big)=0$ for all $\gamma>-s$, this will imply that $\mathcal H^\phi\big(E(-s,f)\big)>0$
hence $\dimh\big(E(-s,f)\big)\geq 1$.

Let us proceed with the construction of the compact set following Section \ref{sec:compactset}. We still assume that $\psi\geq 1$ on the dyadic interval $K_0:=\left[\frac k{2^t},\frac{k+1}{2^t}\right]$
and that $\psi=0$ outside $[0,1]$. Let $(N_n)$ be a nondecreasing sequence of integers with $N_1>t$. We define inductively a decreasing sequence $(K_n)$ of compact subsets of $K_0$ such that
$K_n$ consists of $2^{N_1+\cdots+N_n-nt}$ closed dyadic intervals of width $2^{-(N_1+\cdots+N_n+t)}$ and each of these intervals may be written 
 $\left[\frac{k+l2^t}{2^{N_1+\cdots+N_n+t}},\frac{k+l2^t+1}{2^{N_1+\cdots+N_n+t}}\right]$ for some $l\in\ZZ$.
Let us assume that the construction has been done until $K_n$ and let us construct $K_{n+1}$. Let $\Theta_n$ be the set of closed dyadic intervals of width $2^{-(N_1+\cdots+N_n+t)}$
contained in $K_n$ and let $\lambda\in K_n$, $\lambda=\left[\frac{a}{2^{N_1+\dots+N_n+t}},\frac{a+1}{2^{N_1+\cdots+N_n+t}}\right]$. We define $\Theta_{n+1,\lambda}$ as
the set of the intervals $I_m=\left[\frac{k+m2^t}{2^{N_1+\dots+N_{n+1}+t}},\frac{k+m2^t+1}{2^{N_1+\dots+N_{n+1}+t}}\right]$ contained in $\lambda$, with $m\in\ZZ$. 
Since $I_m\subset\lambda$ if and only if 
$$2^{N_{n+1}-t}a-k2^{-t}\leq m\leq 2^{N_{n+1}-t}a-k2^{-t}+2^{N_{n+1}-t}-2^{-t}$$
there are exactly $2^{N_{n+1}-t}$ such intervals. Thus we may define 
$$K_{n+1}=\bigcup_{\lambda\in\Theta_n}\bigcup_{I\in \Theta_{n+1,\lambda}}I$$
which satisfies our requirements.

We then set $K=\bigcap_{n\geq 0}K_n$ and we prove that if we choose conveniently the sequence $(N_n)$, then $\mathcal H^\phi(K)>0$ where $\phi(s)=s \exp\left(\log^{3/4}\left(\frac 1s\right)\right)$.
Indeed, let $\mu$ be the mass distribution on $K$ so that each interval of $\Theta_n$ has mass $2^{-(N_1+\cdots+N_n-nt)}$. Let $I$ be an interval with small length and $n$ be the integer such that
$$\frac 1{2^{N_1+\cdots+N_{n+1}+t}}\leq |I|\leq \frac 1{2^{N_1+\cdots+N_n+t}}.$$
Then $I$ can intersect at most two of the intervals of $\Theta_n$ so that 
$$\mu(I)\leq \frac{2}{2^{N_1+\cdots+N_n-nt}}\leq \frac{2}{2^{N_1+\cdots+N_{n+1}+t}}\times 2^{N_{n+1}+(n+1)t}.$$
We fix the sequence $(N_n)$ by setting $N_n=n+t$. With this definition, it is easy to see that there exists $C>0$ such that, for all $n\geq 1$ large enough,
$$2^{N_{n+1}+(n+1)t}\leq \exp\left((N_1+\cdots+N_{n}+t)^{3/4}\log^{3/4} 2\right)\leq \exp\left(\log^{3/4}\left(\frac 1{|I|}\right)\right).$$
Therefore, $\mu(I)\leq \phi(|I|)$ and by the mass transference principle (see e.g. \cite[Lemma 3.18]{Dur15}), $\mathcal H^\phi(K)>0$.

We turn to the construction of $f$. For each $\lambda=\left[\frac{k+l2^{t}}{2^{N_1+\cdots+N_n+t}},\frac{k+l2^{t}+1}{2^{N_1+\cdots+N_n+t}}\right]\in\Theta_n$, we
set $\mu(\lambda)=\left[\frac l{2^{N_1+\cdots+N_n}},\frac {l+1}{2^{N_1+\cdots+N_n}}\right]$ and, as in the proof of Proposition \ref{prop:bigcantorset}, we observe that, 
for all $x\in K$, either $\psi_{\mu(\lambda)}(x)\geq 1$ if $x\in\lambda$ or $\psi_{\mu(\lambda)}(x)=0$. We then set
$$f_n=2^{-(N_1+\cdots+N_n)s}\sum_{\lambda\in\Theta_n}\psi_{\mu(\lambda)}$$
which belongs to $B_p^{s,1}(\RR)$ since
$$\|f_n\|\leq 2^{-(N_1+\cdots+N_n)s}2^{(N_1+\cdots+N_n-nt)/p}2^{(N_1+\cdots+N_n)\times\left(s-\frac 1p\right)}\leq 1.$$
We finally set $f=\sum_{n\geq 1}n^{-2}f_n\in B_p^{s,1}(\RR)$. As in the proof of Theorem \ref{thm:saturatingconv1}, it is easy to prove
that $(P_j(x))$ converges for all $x\in\RR^d$. Moreover, for $x\in K$ and $j\in (N_1+\cdots+N_{n-1},N_1+\cdots+N_{n}]$, one knows that
$$R_jf(x)\geq n^{-2}2^{-(N_1+\cdots+N_n)s}.$$
This gives
$$\liminf_j \frac{\log R_jf(x)}{j\log 2}\geq \liminf_n \frac{-s(N_1+\cdots+N_n)}{N_1+\cdots+N_{n+1}}=-s$$
since we have taken a sequence $(N_n)$ which does not increase too fast. Hence, for this function $f$, $K\subset \mathcal E^+(-s,f)$ and we are done.

\section{On the packing dimension of $\mathcal E^+(\beta,f)$}\label{sec:packing}

\subsection{The case $d-sp>0$ and $\beta>0$}
We first prove the first half of Theorem \ref{thm:badpacking} for $\beta>0$.
 We will follow a variant of the construction done in Proposition \ref{prop:bigcantorset}; here we will construct a subset $L$ of $K$ with different Hausdorff and
 packing dimension. Since the Hausdorff dimension of $L$ will be smaller than its packing dimension, we will be able to construct a saturating function $f$ such that,
 at some level $j$, for all $x\in L$, $P_j f(x)$ is bigger than the expected value if we look only at the packing dimension of $L$. Thanks to a very careful
 construction, this property will still hold for all levels $j$, leading to a function $f$ satisfying $  \dimp\left(\mathcal E^+(\beta,f)\right)>1-sp-\beta p$.
 
 As before, we assume that $\psi=\psi^{(1)}\geq 1$ on some
 $\left[\frac{k_0}{2^{t_0}},\frac{k_0+1}{2^{t_0}}\right]$ and that $\psi=0$ outside $[0,1]$. 
%
Let $u,v>1$ and let $N\geq t\geq t_0$ be two integers. 
We then consider $k$ such that $\psi\geq 1$ on $\left[\frac k{2^t},\frac{k+1}{2^t}\right]$ and we define the set $\Omega$ and the similarities $s_m$ as in the proof of Proposition \ref{prop:bigcantorset}.
We define a sequence $(N_k)$ by setting $N_0=1$, $N_{2k+1}=uN_{2k}$ and $N_{2k+2}=vN_{2k+1}$ so that $N_{2k}=(uv)^k$ and $N_{2k+1}=u(uv)^k$. We also define a sequence of compact sets 
$(L_j)$ by setting $L_0=\left[\frac k{2^t},\frac{k+1}{2^t}\right]$ and 
\begin{itemize}
 \item if $j\in [N_{2k},N_{2k+1})$, $L_{j+1}=\bigcup_{m\in\Omega}s_m(L_j)$;
 \item if $j\in [N_{2k+1},N_{2k+2})$, $L_{j+1}=s_1(L_j)$.
\end{itemize}
We finally define $L=\bigcap_j L_j$. It is easy to check that each $L_j$ consists of closed dyadic intervals of width $2^{-(t+Nj)}$. Denote by $\Gamma_j$ the set of these intervals
and by $M_j$ its cardinal number. By construction, $M_0=1$, $M_{N_{2k+2}}=M_{N_{2k+1}}$ whereas $M_{j+1}=2^{N-t}M_j$ provided $j$ belongs to $[N_{2k},N_{2k+1})$. An elementary computation
shows that
\begin{eqnarray*}
 M_{N_{2k}}&=&2^{(N-t)(u-1)\frac{(uv)^k-1}{uv-1}}\\
 M_{N_{2k+1}}&=&2^{(N-t)(u-1)\frac{(uv)^{k+1}-1}{uv-1}}.
\end{eqnarray*}
By the results of \cite{Baek06} on the dimension of homogeneous Cantor sets,
\begin{eqnarray*}
 \dimp(L)&=&\limsup_j \frac{\log M_j}{Nj \log 2}\\
 &=&\limsup_k \frac{\log M_{N_{2k+1}}}{N N_{2k+1}\log 2}\\
 &=&\frac{N-t}N\times\frac{(u-1)v}{uv-1}.
\end{eqnarray*}
Observe also, even if this will not be required for the sequel, that
\begin{eqnarray*}
 \dimh(L)&=&\liminf_j \frac{\log M_j}{Nj \log 2}\\
 &=&\limsup_k \frac{\log M_{N_{2k}}}{N N_{2k}\log 2}\\
 &=&\frac{N-t}N\times\frac{(u-1)}{uv-1}=\frac{\dimp(L)}{v}.
\end{eqnarray*}
We are now ready to construct the function $f$. For $l\geq 1$, define $c_{Nl}=2^{\frac{Nl}p(1-sp)}M_l^{-1/p}$ and 
$f_l=c_{Nl}\sum_{\lambda\in \Gamma_{l}}\psi_{\mu(\lambda)}$
so that $\|f_l\|\leq 1$. Recall that the construction of the sets $\mu(\lambda)$ together with that of the similarities $s_m$ ensure that,
for any $x\in L$, $f_l(x)\geq c_{Nl}$. As usual, $f\in B_p^{s,1}(\RR)$ is defined by $f=\sum_{l\geq 1}l^{-2}f_l$.

We shall control $\log P_j f(x)/j\log 2$ for all $x\in L$ and all $j\geq 1$. We fix $\eta>0$ and assume first that $j$ belongs to some $(NN_{2k+1}(1+\eta),N N_{2k+2}]$. In that
case, $j\in \big(Nl,N(l+1)\big]$ with $l=N_{2k+1}(1+\kappa)$ and $\kappa\in [\eta,v-1]$. Since in that case
$$M_l=M_{N_{2k+1}}=2^{(N-t)(u-1)\frac{(uv)^{k+1}-1}{uv-1}},$$
we get for all $x\in L$, 
\begin{eqnarray*}
 \frac{\log P_j f(x)}{j\log 2}&\geq&\frac{\log (c_{Nl}/l^2)}{j\log 2}\\
 &\geq&\frac{N N_{2k+1}(1+\kappa)(1-sp)-(N-t)(u-1)\frac{(uv)^{k+1}-1}{uv-1}}{pN N_{2k+1}(1+\kappa)}+o(1).
\end{eqnarray*}
Remembering that $N_{2k+1}=u(uv)^k$, we deduce
\begin{eqnarray*}
 \liminf_{\substack{j\to+\infty\\ j\in\bigcup_k [NN_{2k+1}(1+\eta),N N_{2k+2})}}\frac{\log P_j f(x)}{j\log 2}&\geq&\frac1p\times\left(1-sp-\frac{N-t}N\times\frac{(u-1)v}{(uv-1)(1+\eta)}\right)
\end{eqnarray*}

Assume now that $j$ belongs to some $(N N_{2k}, N N_{2k+1}(1+\eta)]$. In that case, we use that $P_j f(x)\geq c_{NN_{2k}}/N_{2k}^2$ to get
\begin{eqnarray*}
 \frac{\log P_j f(x)}{j\log 2}&\geq&\frac{N N_{2k}(1-sp)-(N-t)(u-1)\frac{(uv)^{k}-1}{uv-1}}{p N N_{2k+1}(1+\eta)}+o(1)
\end{eqnarray*}
so that
\begin{eqnarray*}
 \liminf_{\substack{j\to+\infty\\ j\in\bigcup_k [N N_{2k},NN_{2k+1}(1+\eta))}}\frac{\log P_j f(x)}{j\log 2}&\geq&
 \frac1{p(1+\eta)u}\times\left(1-sp-\frac{N-t}N\times\frac{u-1}{(uv-1)}\right)\\
\end{eqnarray*}

It is time now to choose $N$, $t$, $u$, $v$ and $\eta$ so that $\dimp(L)>1-sp-\beta p$ and, for all $x\in L$, $\liminf_j \log P_jf(x)/j\log 2\geq\beta$.
The real number $\beta\in\left(0,\frac 1p-s\right)$ being fixed, and using the change of variables $\alpha=1-sp-\beta p$, we are done if we may choose the parameters so that
\begin{align}
\label{eq:badpacking1}
\frac{N-t}{N}\times\frac{(u-1)v}{uv-1}&>\alpha\\
\label{eq:badpacking2}
1-sp-\frac{N-t}{N}\times\frac{(u-1)v}{(uv-1)(1+\eta)}&\geq 1-sp-\alpha\\
\label{eq:badpacking3}
\frac{1}{(1+\eta)u}\left(1-sp-\frac{N-t}{N}\times\frac{u-1}{uv-1}\right)&\geq 1-sp-\alpha.
\end{align}
Let $\veps>0$ and set $u=1+\veps$ and $v=1+\left(\frac1\alpha -1\right)\veps$. It is easy
to check that $(u-1)v/(uv-1)>\alpha$. Since $\left\{\frac{N-t}{N};\ N\geq t\geq t_0\right\}$ is dense in $(0,1)$, we may find two integers $N\geq t\geq t_0$ such that
\begin{equation*}
 \alpha\left(1+\veps^2\right)\geq\frac{N-t}{N}\times\frac{(u-1)v}{uv-1}>\alpha.
\end{equation*}
The right part of this inequality is \eqref{eq:badpacking1}. We finally choose $\eta>0$ such that 
\[\frac{N-t}{N}\times\frac{(u-1)v}{(uv-1)(1+\eta)}=\alpha.\]
This implies that \eqref{eq:badpacking2} is true and that $\eta=o(\veps)$. It remains to justify that \eqref{eq:badpacking3} is verified provided $\veps>0$ is small enough. Now
\begin{align*}
\frac{1}{(1+\eta)u}\left(1-sp-\frac{N-t}{N}\times\frac{u-1}{uv-1}\right)&\geq 
\frac{1}{1+\veps+o(\veps)}\left(1-sp-\frac{\alpha(1+\veps^2)}{1+\left(\frac 1\alpha -1\right)\veps}\right)\\
&\geq 1-sp-\alpha+sp\veps+o(\veps)\\
&\geq 1-sp-\alpha
\end{align*}
for small values of $\veps$. Observe the role of the assumption $s>0$ in the last line (Theorem \ref{thm:badpacking} is false when $s=0$ if we are working
with the Haar basis).

\subsection{The case $d-sp>0$ and $\beta<0$}
The proof of the case $\beta<0$ and still $1-sp>0$ of Theorem \ref{thm:badpacking} follows the same line. 
We do exactly the same construction for the compact set $L$ and for the function $f$. There is an additional difficulty now: 
we have to verify that the wavelet series is convergent at each point of $L$. This will be true provided there exists $\delta>0$ such that,
for all $l\in\mathbb N$, $c_{Nl}\leq 2^{-\delta Nl}$. The worst case (corresponding to the biggest values of $c_{Nl}$) corresponds to the case $l=N_{2k}$. In that case
\begin{align*}
c_{NN_{2k}}&=2^{\frac 1p\left((1-sp)N(uv)^k-(N-t)(u-1)\frac{(uv)^k-1}{uv-1}\right)}\\
&=2^{\frac{NN_{2k}}p\left(1-sp-\frac{(N-t)(u-1)}{N(uv-1)}+o(1)\right)}.
\end{align*}
Therefore, we will need the condition 
\begin{equation}
1-sp-\frac{N-t}N\times\frac{u-1}{uv-1}<0. \label{eq:badpacking4}
\end{equation}
Another difference with the previous case is that we are looking at the remainders instead of the partial sums. When evaluating $R_j f(x)$, we can now use $c_l$ for $l\geq j$ instead of $l<j$. 
Hence we have to cut the intervals $[NN_{2k},NN_{2k+2})$ in a different way. We still consider $\eta>0$ and assume first that $j$ belongs to some $[NN_{2k},N(1-\eta)N_{2k+1})$. In that case, $j\in [Nl,N(l+1))$ with $l=\kappa N_{2k}$ and
$\kappa\leq (1-\eta)u$. Moreover, we know that for these values of $l$,
$$M_l=2^{(N-t)(\kappa-1)(uv)^k+(N-t)(u-1)\frac{(uv)^k-1}{uv-1}}.$$
This yields that for all $x\in L$,
\begin{align*}
\frac{\log R_j f(x)}{j\log 2}&\geq \frac{\log c_{N(l+1)}}{j\log 2}+o(1)\\
&\geq \frac{N\kappa (uv)^k(1-sp)-(N-t)(\kappa-1)(uv)^k+(N-t)(u-1)\frac{(uv)^k-1}{uv-1}}{pN\kappa (uv)^k}+o(1)\\
&\geq \frac 1p\left(1-sp-\frac{N-t}N\left(1-\frac{u(v-1)}{\kappa(uv-1)}\right)\right)+o(1).
\end{align*}
The lower bound of the right handside of this inequality is attained for the largest possible value of $\kappa$, namely for $\kappa=(1-\eta)u$ so that
\[\liminf_{\substack{j\to+\infty\\ j\in\bigcup_k [N N_{2k},N(1-\eta)N_{2k+1})}}\frac{\log R_j f(x)}{j\log 2}\geq\frac 1p\left(1-sp-\frac{N-t}N\left(1-\frac{v-1}{(1-\eta)(uv-1)}\right)\right).\]

On the other hand, for $j$ belonging to $[N(1-\eta)N_{2k+1},NN_{2k+2})$, we look at a term later in the series by writing $R_j f(x)\geq c_{NN_{2k+2}}$ so that 
\begin{align*}
\frac{\log R_j f(x)}{j\log 2}&\geq \frac{N(uv)^{k+1}(1-sp)-(N-t)(u-1)\frac{(uv)^{k+1}-1}{uv-1}}{p(1-\eta)u(uv)^k N}+o(1)\\
&\geq \frac 1p\times\frac{v}{1-\eta}\times\left(1-sp-\frac{N-t}N\times\frac{u-1}{uv-1}\right)+o(1).
\end{align*}

Hence, we are done provided we may choose the parameters so that \eqref{eq:badpacking4} and the three following inequalities are satisfied:
\begin{align}
\label{eq:badpacking5}
\frac{N-t}N\times\frac{(u-1)v}{uv-1}&>\alpha\\
\label{eq:badpacking6}
1-sp-\frac{N-t}{N}\left(1-\frac{v-1}{(1-\eta)(uv-1)}\right)&\geq 1-sp-\alpha\\
\label{eq:badpacking7}
\frac{v}{1-\eta}\times\left(1-sp-\frac{N-t}N\times\frac{u-1}{uv-1}\right)&\geq 1-sp-\alpha.
\end{align}

As before, we consider $\veps>0$ very small and set $u=1+\veps$, $v=1+\left(\frac 1\alpha-1\right)\veps$, $N\geq t\geq t_0$ so that 
\begin{equation}
\label{eq:badpacking8}
\alpha(1+\veps^2)\geq\frac{N-t}{N}\times\frac{(u-1)v}{uv-1}>\alpha.
\end{equation}
This ensures that \eqref{eq:badpacking5} is true and also that \eqref{eq:badpacking4} is satisfied provided
$\veps>0$ is small enough: remember that $\beta=(1-sp-\alpha)/p<0$ and that $\frac{N-t}N\times \frac{u-1}{uv-1}$ can be taken arbitrarily close to $\alpha$.  We now set $\eta=\veps^{3/2}$ and we claim that 
\eqref{eq:badpacking6} and \eqref{eq:badpacking7} are also satisfied. Indeed, we write
\begin{align*}
1-\frac{v-1}{(1-\eta)(uv-1)}&=1-\frac{v-1}{uv-1}-\frac{v-1}{uv-1}\veps^{3/2}+o(\veps^{3/2})\\
&=\frac{(u-1)v}{uv-1}-\frac{(u-1)v}{uv-1}\times\frac{v-1}{v(u-1)}\veps^{3/2}+o(\veps^{3/2})
\end{align*}
so that, using also \eqref{eq:badpacking8},
\begin{align*}
\frac{N-t}N\left(1-\frac{v-1}{(1-\eta)(uv-1)}\right)&\leq \alpha-\alpha\left(\frac1\alpha-1\right)\veps^{3/2}+o(\veps^{3/2})\\
&\leq \alpha
\end{align*}
provided $\veps$ is small enough. Moreover,
\begin{align*}
\frac{v}{1-\eta}\left(1-sp-\frac{N-t}{N}\times\frac{u-1}{uv-1}\right)&=\frac{v}{1-\eta}(1-sp)-\frac{N-t}N\times\frac{(u-1)v}{uv-1}\times\frac1{1-\eta}\\
&\geq(1-sp)\left(1+\left(\frac1\alpha-1\right)\veps+o(\veps)\right)-\alpha+o(\veps)\\
&\geq(1-sp-\alpha)+(1-sp)\left(\frac 1\alpha-1\right)\veps+o(\veps)\\
&\geq1-sp-\alpha
\end{align*}
provided again that $\veps>0$ is small enough. Observe the role played here by the assumption $1-sp>0$.

\subsection{The case $d-sp<0$} In that case, which implies that $(P_jf(x))$ converges for all $x\in\RR^d$, we are able to prove that for all $f\in B_{p}^{s,\infty}(\RR^d)$, 
$\dimp\big(\mathcal E^+(\beta,f)\big)\leq d-sp-\beta p$ for all $\beta\in \left[-s,\frac dp-s\right]$.
Let $A>0$ be such that all mother wavelets have support in $[-A,A]^d$. For $j\geq 1$ and $x\in\RR^d$, we denote
\[\Gamma_j(x)=\left\{\lambda\in\Lambda_j;\ \exists i,\ \psi_\lambda^{(i)}(x)\neq 0\right\}.\]
The cardinal number of $\Gamma_j(x)$ is uniformly bounded in $j$ and $x$ (by $(2A+1)^d$). We will need another combinatorial result.
\begin{lemma}\label{lem:packingbest1}
Let $l\geq 1$, $\lambda\in\Lambda_l$ and $(x_u)$ a sequence in $\RR^d$. Then
\[\mathrm{card}\left(\left\{u\in\NN;\ \lambda\in\Gamma_l(x_u)\right\}\right)\leq
\sup_u\left(\mathrm{card}\left(\left\{v\in\NN;\ \|x_u-x_v\|\leq 2A2^{-l}\right\}\right)\right).\]
\end{lemma}
\begin{proof}
Assume that $u$ and $v$ are such that $\lambda\in \Gamma_l(x_u)$ and $\lambda\in\Gamma_l(x_v)$. Then we have simultaneously $2^l x_u-k\in [-A,A]^d$ and $2^l x_v-k\in [-A,A]^d$ so that $\|x_u-x_v\|\leq 2A2^{-l}$.
\end{proof}

The forthcoming lemma is a substitute to Lemma \ref{lem:packing1}.
\begin{lemma}\label{lem:packingbest2}
Let $\veps>0$ and $f\in B_p^{s,\infty}(\RR^d)$. There exists $C=C_\veps$ so that, for all $x\in\RR^d$, for all $j\geq 1$,
\[\left|R_jf(x)\right|\leq C\left(\sum_{l\geq j}2^{\veps pl}\sum_i\sum_{\lambda\in\Gamma_l(x)}|c_\lambda^{(i)}|^p\right)^{1/p}.\]
\end{lemma}
\begin{proof}
We use Hölder's inequality and the fact that the cardinal number of $\Gamma_l(x)$ is uniformly bounded in $x$ and $l$ to get successively
\begin{align*}
|R_jf(x)|&\leq C\sum_{l\geq j}\sum_i\sum_{\lambda\in\Gamma_l(x)}|c_{\lambda}^{(i)}|\\
&\leq C\sum_{l\geq j}\sum_i\sum_{\lambda\in\Gamma_l(x)}2^{\veps l}|c_\lambda^{(i)}|2^{-\veps l}\\
&\leq C\left(\sum_{l\geq j}2^{\veps pl}\sum_i\sum_{\lambda\in\Gamma_l(x)}|c_\lambda^{(i)}|^p\right)^{1/p}\left(\sum_{l\geq j}2^{-\veps p^* l}\sum_i \mathrm{card}(\Gamma_l(x))\right)^{1/p^*}\\
&\leq C\left(\sum_{l\geq j}2^{\veps pl}\sum_i\sum_{\lambda\in\Gamma_l(x)}|c_\lambda^{(i)}|^p\right)^{1/p}.
\end{align*}
\end{proof}
The proof that, for $\beta\in\left[-s,\frac dp-s\right]\backslash\{0\}$ and $f\in B_p^{s,\infty}(\RR^d)$, $\dimp\left(\mathcal E^+(\beta,f)\right)\leq d-sp-\beta p$ 
follows the same line as the proof of Proposition \ref{prop:packing1} with some technical changes. As before, letting
\[\mathcal G_J^+(\gamma,f)=\left\{x\in\RR^d;\ \forall j\geq J,\ |R_jf(x)|\geq 2^{\gamma j}\right\}\]
for $\gamma\in \left(\beta,\frac dp-s\right)$,
one only need to prove that, for all $J\in\NN$, $\dboxsup\left(\mathcal G_J^+(\gamma,f)\right)\leq d-sp-\gamma p$.
Let $j\geq J$ and let $\Theta_j$ be the dyadic cubes of the $j$-th generation intersecting $\mathcal G_J^+(\gamma,f)$. 
Let $N_j$ be the cardinal number of $\Theta_j=\left\{\lambda_1,\dots,\lambda_{N_j}\right\}$.
Pick $x_u\in \lambda_u\cap\mathcal G_J^+(\gamma,f)$ for $u=1,\dots,N_j$. Then by Lemma \ref{lem:packingbest2}, 
$$2^{\gamma pj}N_j\leq C\sum_{l\geq j}2^{\veps pl}\sum_i \sum_{u=1}^{N_j}\sum_{\lambda\in \Gamma_l(x_u)}|c_\lambda^{(i)}|^p.$$
Now, the $x_u$ belonging to different dyadic cubes of the $j$-th generation, for all $l\geq j$, 
$$\sup_u\left(\mathrm{card}\left(\left\{v;\ \|x_u-x_v\|\leq 2A2^{-l}\right\}\right)\right)\leq C_{A,d}.$$
Therefore, an application of Lemma \ref{lem:packingbest1} yields
\begin{align*}
2^{\gamma pj}N_j&\leq C\sum_{l\geq j}2^{\veps pl}\sum_i \sum_{\lambda\in\Lambda_l}|c_\lambda^{(i)}|^p\\
&\leq C\sum_{l\geq j}2^{\veps pl}\left(\sum_{i}\sum_{\lambda\in\Lambda_l} |c_\lambda^{(i)}|^p2^{(sp-d)l}\right)2^{-(sp-d)l}\\
&\leq C\sum_{l\geq j}2^{(d-sp+\veps p)l}\|f\|^p_{B_p^{s,\infty}}.
\end{align*}
Since $d-sp<0$, we may choose $\veps>0$ sufficiently small so that $d-sp+\veps p<0$. We deduce that
$$2^{\gamma pj}N_j\leq C2^{(d-sp+\veps p)j}\sum_{l\geq j}2^{(d-sp+\veps p)(l-j)}\|f\|^p_{B_p^{s,\infty}}$$
which in turn implies that
$$N_j\leq C 2^{(d-sp-\gamma p+\veps p)j} \|f\|^p_{B_p^{s,\infty}}$$
yielding $\dboxsup\left(\mathcal G_J^+(\gamma,f)\right)\leq d-sp-\gamma p+\veps p$. Letting $\veps$ to 0 implies the result.

\begin{question}
 Does this remain true if we do not assume that the wavelets have compact support?
\end{question}

\section{Final remarks}\label{sec:remarks}

\subsection{Residuality and prevalence}
In Theorem \ref{thm:main1}, we cannot expect to get that the set of functions satisfying (iii) is residual.
In fact, we are very far from this, as the following proposition indicates.
\begin{proposition}\label{prop:propositionlimit}
 Let $X=\besov$  or $X=\sobolev$ with $p,q\in[1,+\infty)$. Then for all functions $f$ in a residual subset of $X$, for all $x\in\mathbb R^d$, 
 \begin{equation}
  \liminf_{j\to+\infty}\frac{\log^+|P_j f(x)|}{j\log 2}=0. \label{eq:propositionlimit}
 \end{equation}
\end{proposition}
\begin{proof}
 For $K$ a compact subset of $\RR^d$, $\veps>0$ and $J\in\NN$, we denote by 
 $$\mathcal U(K,\veps,J)=\left\{f\in X;\ \forall x\in K,\ \exists j\geq J,\ |P_jf(x)|<2^{\veps j}\right\}.$$
 Then all $\mathcal U(K,\veps,J)$ are dense (because they contain all functions with a finite wavelet series) and open. 
Indeed, pick $f\in\mathcal U(K,\veps,J)$.  For any $x\in K$, there exists $j\geq J$ such that $|P_j f(x)|<2^{\veps j}$.
By continuity of $(g,y)\mapsto P_j g(y)$, there exists an open neighbourhood $\mathcal O_x$ of $x$ in $K$ and a neighbourhood $\mathcal V_x$ of $f$ in $X$ such that
$$\forall g\in\mathcal V_x,\ \forall y\in \mathcal O_x,\ |P_j g(y)|<2^{\veps j}.$$
By compactness, $K$ is covered by a finite number of open sets $\mathcal O_x$, says $\mathcal O_{x_1},\dots,\mathcal O_{x_p}$.
Then $\mathcal V_{x_1}\cap\dots\cap \mathcal V_{x_p}$ is a neighbourhood of $f$ contained in $\mathcal U(K,\veps,J)$.
We conclude by observing that, if $(K_m)$ is a sequence of compacts subsets of $\RR^d$ such that $\bigcup_m K_m=\RR^d$, 
any function $f$ in the residual set $\bigcap_{m,k,J}\mathcal U(K_m,2^{-k},J)$ satisfies \eqref{eq:propositionlimit} for all $x\in\mathbb R^d$.
\end{proof}

On the contrary, we do not know whether we can get the existence of a prevalent set of strongly multifractal functions. This can be done
if we modify the definition of our sets by taking absolute values. More precisely, let us define
\begin{align*}
 P_j^*f(x)&=\sum_{l<j}|Q_l f(x)|\\
 R_j^*f(x)&=\sum_{l\geq j}|Q_lf(x)|
\end{align*}
and, for $\beta>0$, 
$$E_*(\beta,f)=\left\{x\in\RR^d;\ \lim_j \frac{\log P_j^*f(x)}{j\log 2}=\beta\right\},$$
for $\beta<0$, 
$$E_*(\beta,f)=\left\{x\in\RR^d;\ \lim_j \frac{\log R_j^*f(x)}{j\log 2}=\beta\right\}.$$

\begin{theorem}\label{thm:prevalenceabsolutevalue}
Let $s\geq 0$, $p,q\in [1,+\infty)$ and $X=\besov$ or $X=\sobolev$. Assume that the wavelets have compact support.
 \begin{itemize}
  \item[(i)] For all $f\in X$, for all $\beta\in \left[-s,\frac dp-s\right]\backslash\{0\}$,
  \begin{eqnarray*}
      \dimp\big( E_*(\beta,f)\big)&\leq&d-sp-\beta p.
  \end{eqnarray*}
\item[(ii)] For all  $f$ in a prevalent subset of $X$, for all $\beta\in\left[-s,\frac dp-s\right]\backslash\{0\}$,
$$\dimh\big(E_*(\beta,f)\big)=\dimp\big(E_*(\beta,f)\big)=d-sp-\beta p.$$
 \end{itemize}
\end{theorem}

\begin{proof}
 The proof of Section \ref{sec:dimension} works mutatis mutandis for the sets $E_*(\beta,f)$, thus there is nothing to do to prove (i). To prove (ii), we need the following lemma.
 \begin{lemma}\label{lem:prevalenceabsolutevalue}
  Let $x\in\RR^d$, $f\in B_p^{s,1}(\RR^d)$, $M>0$. Assume that there exists $j_0\in\NN$ such that, for all $j\geq j_0$, there exists $m\in\{0,\dots,M-1\}$
  such that $P_{jM+m}^*f(x)\geq 2^{\beta jM}$. Then
  \[\liminf_j \frac{\log P_j^*f(x)}{j\log 2}\geq \beta .\]
 \end{lemma}
 We postpone the proof of this lemma to finish that of Theorem \ref{thm:prevalenceabsolutevalue}. We need an analogue of Lemma \ref{lem:prevalencesaturatingdiv} where we replace
 $P_j$ par $P_j^*$ and $ \limsup_j$ by $\liminf_j$. We keep the notations of this lemma . In particular its proof
 (see more particularly Remark \ref{rem:prevalenceprecised}) shows that, for all $G\subset K$ with $\dboxsup(G)<\alpha$, for all $\beta<(d-sp-\alpha)/p$, 
 there exists $J\geq 1$ such that, for all $f$ in a prevalent subset of $X$, for all $x\in G$, 
 $$\liminf_{j\to+\infty}\sup_{k=0,\dots,jN-1}\frac{\log P_{jJN+k}^* f(x)}{jJN\log 2}\geq\beta.$$
 Lemma \ref{lem:prevalenceabsolutevalue} tells us that in fact
 $$\liminf_{j\to+\infty}\frac{\log P_j^* f(x)}{j\log 2}\geq \beta.$$
 We now conclude with a proof that imitates strongly that done before.
\end{proof}
\begin{proof}[Proof of Lemma \ref{lem:prevalenceabsolutevalue}]
 Let $l\geq (j_0+1)M$. There exists $j\geq j_0$ such that $l\in [(j+1)M,(j+2)M)$. For this value of $j$, we know the existence
 of $m\in\{0,\dots,M-1\}$ satisfying
 \[P_{jM+m}^* f(x)\geq 2^{\beta jM}.\]
 Since the sequence $(P_n^*(x))$ is nondecreasing, 
 $$P_l^*f(x)\geq 2^{\beta jM}\geq \frac1{2^{2\beta M}}\cdot 2^{\beta l}.$$
\end{proof}

These considerations suggest that prevalence
is a more suitable notion of genericity than residuality in the context of multifractal analysis.

\subsection{Extreme values for p and/or q}
When $p\in(0,1)$ or $q\in (0,1)$, the Besov spaces $B_p^{s,q}(\RR^d)$ are no more Banach spaces but nonetheless are separable complete metric vector spaces. Our method of proof carries on without difficulties to this context. The only important change is that we can no longer apply H\"older's inequality during the proof of Lemma \ref{lem:packing1} when $p<1$. But the proof is even simpler. We just write
\[\sum_i\sum_{\lambda\in\Lambda_{l,\lambda_j(x),\veps}}|c_\lambda^{(i)}|\leq \left(\sum_i\sum_{\lambda\in\Lambda_{l,\lambda_j(x),\veps}}|c_\lambda^{(i)}|^p\right)^{1/p}.\]

When $p=+\infty$ or $q=+\infty$, the Besov spaces are no longer separable. Part (i) and (iii) of Theorems \ref{thm:main1} and \ref{thm:main2}, which do not use separability, remain valid.
However, we do not know if this the case for part (ii). 

\subsection{The sets $E(0,f)$}
Our work does not consider the case $\beta=0$. It seems natural to define the corresponding sets as, for instance
\begin{align*}
 E(0,f)=\Bigg\{x\in\RR^d;\ &\big(P_jf(x)\big)\textrm{ diverges and }\lim_j \frac{\log |P_jf(x)|}{j\log 2}=0\\
 &\textrm{or }\big(P_jf(x)\big)\textrm{ converges and }\lim_j \frac{\log |R_jf(x)|}{j\log 2}=0\Bigg\}.
\end{align*}
The construction of the saturating function may also be done for this set and an easy modification of the proof of Proposition \ref{prop:dimh} shows
that $\dimh\big(\mathcal E^-(0,f)\big)\leq d-sp$. However, using ideas from Section \ref{sec:packing}, this breaks down for the packing dimension.

\begin{proposition}
 Assume that the wavelets have compact support and $1-sp>0$. Then, for all $\veps\in(0,1)$, there exists $f\in B_p^{s,1}(\RR)$ such that $\dimp\big(E(0,f)\big)\geq1-\veps$.
\end{proposition}
 \begin{proof}
  Let $L$ be the compact set built in Section \ref{sec:packing} with $v=1/(1-sp)>1$ and $u,N,T$ such that 
  $$(1-\veps)(1-sp)<\frac{u-1}{uv-1}\times\frac{N-t}N<1-sp$$
  (this is always possible by taking $u,N,t$ large enough). Hence,
  $$\dimp(L)={v}\times \frac{u-1}{uv-1}\times\frac{N-t}N>1-\veps.$$
  Let, for $k\geq 0$, $c_{NN_{2k}}=2^{NN_{2k}/(k+1)}$ and $f_k=c_{NN_{2k}}\sum_{\lambda\in \Gamma_{N_{2k}}}\psi_{\mu(\lambda)}$. Then
  \begin{align*}
   \|f_k\|&\leq2^{NN_{2k}/(k+1)}\card(\Gamma_{N_{2k}})^{1/p}2^{-(1-sp)NN_{2k}/p}\\
   &\leq 2^{NN_{2k}/(k+1)}2^{\frac{NN_{2k}}p\left(-(1-sp)+\frac{(N-t)(u-1)}{N(uv-1)}+o(1)\right)}\\
   &\leq 2^{NN_{2k}\left(-(1-sp)+\frac{(N-t)(u-1)}{N(uv-1)}+o(1)\right)}.
  \end{align*}
 This ensures that $\|f_k\|\leq C$ so that $f=\sum_{k\geq 0}f_k/(k+1)^2$ defines an element of $B_p^{s,1}(\RR)$. For all $x\in L$ and all $j\in (NN_{2k},NN_{2k+1}]$,
 $P_j f(x)\geq P_{NN_{2k}}f(x)\geq c_{NN_{2k}}/(k+1)^2$. Hence the sequence $(P_jf(x))$ tends to infinity and in particular, $\liminf_j \log |P_jf(x)|/j\log 2\geq 0$.
 Moreover, since the wavelets have compact support, there exists $A>0$ such that,
 for all $x\in L$ and all $k\geq 1$, $|f_k(x)|\leq A c_{NN_{2k}}=A 2^{N(uv)^k/(k+1)}$. Let $j\geq 1$ and $k\geq 0$ be such that $j\in  (NN_{2k},NN_{2k+1}]$. 
 Then
\begin{align*}
 |P_jf(x)|&\leq\sum_{l=0}^k |f_l(x)|\leq A\left(2^{N(uv)^0}+\cdots+2^{N(uv)^k/(k+1)}\right)\\
 &\leq Ck 2^{N(uv)^k/k}.
\end{align*}
This implies clearly that $\limsup_j  \log |P_jf(x)|/j\log 2\leq 0$, hence that $L\subset E(0,f)$.
 \end{proof}
\begin{remark}
 The proof shows that there exists $f\in B_p^{s,1}(\RR)$ such that 
 $$\dimp\left(\left\{x\in\RR^d;\ \big(P_jf(x)\big)\textrm{ diverges and }\lim_j \frac{\log |P_jf(x)|}{j\log 2}=0\right\}\right)\geq 1-\veps.$$
 If we change the definition of $c_{NN_{2k}}$ into $c_{NN_{2k}}=2^{-NN_{2k}/(k+1)}$, we can also prove the existence of $f\in B_p^{s,1}(\RR)$ such that
  $$\dimp\left(\left\{x\in\RR^d;\ \big(P_jf(x)\big)\textrm{ converges and }\lim_j \frac{\log |R_jf(x)|}{j\log 2}=0\right\}\right)\geq 1-\veps.$$
We can also compare these statements with Proposition \ref{prop:dimhdv}.
\end{remark}

\bibliographystyle{plain} 
\bibliography{biblio,bibliomoi,preprint} 


\end{document}